\documentclass[12pt]{amsart}
\usepackage{a4wide}
\usepackage[utf8]{inputenc}
\usepackage{amsmath}
\usepackage{amsthm}
\usepackage{amssymb}
\usepackage{amsfonts}
\usepackage{amsopn}
\usepackage{graphicx}
\usepackage{enumerate}
\usepackage{color}
\usepackage{mathtools}
\usepackage{mathrsfs}
\usepackage{bbm}
\usepackage{yhmath}
\usepackage{cite}
\usepackage[colorlinks,linkcolor=blue,anchorcolor=blue,citecolor=blue]{hyperref}

\newtheorem{theorem}{Theorem}[section]
\newtheorem{proposition}[theorem]{Proposition}
\newtheorem{lemma}[theorem]{Lemma}
\newtheorem{definition}[theorem]{Definition}
\newtheorem{corollary}[theorem]{Corollary}

\theoremstyle{definition}

\numberwithin{equation}{section}

\newcommand{\mbf}{\mathbf}
\newcommand{\mcal}{\mathcal}

\newcommand{\set}[1]{\left\{ #1 \right\}}

\newcommand{\bu}{\boldsymbol{\alpha}}

\newcommand{\R}{\mathbb{R}}

\newcommand{\Z}{\mathbb{Z}}
\newcommand{\N}{\mathbb{N}}

\newcommand{\f}{\infty}

\newcommand{\wh}[1]{\widehat{#1}}

\newcommand{\sse}{\subseteq}

\newcommand{\D}{\;\mathrm{d}}

\newcommand{\bp}{\mathbf{p}}
\newcommand{\bw}{\boldsymbol{\omega}}
\newcommand{\bn}{\mathbf{n}}
\newcommand{\bv}{\boldsymbol{\beta}}

\newcommand{\rmnum}[1]{\romannumeral #1}

\newcommand{\red}[1]{{\color{red}#1}}

\title[Spectrality and eigen sets of infinite convolutions and random measures]{Spectrality and eigen sets of infinite convolutions and random measures generated by admissible pairs}

\author[J. J. Miao]{Jun Jie Miao}
\address[J. J. Miao]{School of Mathematical Sciences,  Key Laboratory of MEA (Ministry of Education) \& Shanghai Key Laboratory of PMMP,  East China Normal University, Shanghai 200241, China}
\email{jjmiao@math.ecnu.edu.cn}

\author[H. Zhao]{Hongbo Zhao}
\address[H. Zhao]{School of Mathematical Sciences, Shanghai Key Laboratory of PMMP, East China Normal University, Shanghai 200241,
	People's Republic of China}
\email{2504245357@qq.com}


\begin{document}
	
	
\maketitle

\begin{abstract}

In this paper, we construct a class of random measures $\mu^{\mathbf{n}}$ by infinite convolutions. Given admissible pairs $\{(N_{k}, B_{k})\}_{k=1}^{m}$ and a  sequence  $\bn=\{n_{k}\}_{k=1}^{\infty}$ of  positive integers, for every $\bw\in \Omega$, we write $\mu^{\mathbf{n}}(\bw) =  \delta_{N_{\omega_{1}}^{-n_{1}}B_{\omega_{1}}} * \delta_{N_{\omega_{1}}^{-n_{1}}N_{\omega_{2}}^{-n_{2}}B_{\omega_{2}}} * \cdots$.  
First, we show that the mapping $\mu^{\mathbf{n}}: (\bw, B) \mapsto \mu^{\mathbf{n}}(\bw)(B)$ is a random measure. 
Next, we introduce the notion of a $t$-equi-positive family, and use it to obtain a general sufficient condition under which an infinite convolution is a spectral measure and possesses a specified set of spectral eigenvalues.  
We then extend the concepts of spectrality and spectral eigenvalues to
random measures, and show that, under the assumption that the corresponding
standard infinite convolution is non-degenerate for $\mathbb{P}$-a.e.\
$\boldsymbol{\omega}\in\Omega$, the measures $\mu^{\mathbf{n}}$ are spectral
random measures for $\mathbb{P}$-a.e.\ $\boldsymbol{\omega}$, admitting the
spectral eigen set
\[
\mathcal{E}_m=\{t\in\mathbb{N}_+:\gcd(t,N_k)=1,\,1\le k\le m\}.
\]
Moreover, for each such $t$, there exist uncountably many spectra $\Lambda_{\boldsymbol{\omega}}\subset\mathbb{Z}$ with $t\Lambda_{\boldsymbol{\omega}}$ also a spectrum of $\mu^{\mathbf{n}}(\boldsymbol{\omega})$.  
Finally, for the important case where each digit set $B_k$ is a consecutive set $\{0,1,\dots,b_k-1\}$, we completely characterise the positive integer spectral eigenvalues of $\mu^{\mathbf{n}}(\boldsymbol{\omega})$, proving that they are exactly the integers coprime to every $b_k$.   

\end{abstract}

\section{Introduction}
\subsection{Infinite convolutions and random measures}

For a finite subset $D \sse \R$, the uniform discrete measure supported on $D$ is given by
\begin{equation*}
	\delta_D = \frac{1}{\# D} \sum_{a \in D} \delta_a,
\end{equation*}
where  $\#$ denotes the cardinality of a set and $\delta_a$ is the Dirac measure at the point $a$. Let $\{ D_k\}_{k=1}^\f$ be a sequence of finite subsets of $\R$ such that $\# D_k \ge 2$ for every $k \ge 1$.
For  each integer $k \geq 1$, we define
\begin{equation*}\label{discrete-convolution}
	\nu_k =\delta_{D_1}*\delta_{D_2} * \cdots *\delta_{D_k},
\end{equation*}
where $*$ denotes the convolution of measures.
If the sequence of convolutions $\{\nu_k\}_{k=1}^\f$ converges weakly to a Borel probability measure $\nu$, then we call $\nu$ the \emph{infinite convolution} of $\{{D_k}\}_{k=1}^\infty$, denoted by
\begin{equation*}\label{def_ica}
	\nu =\delta_{D_1}*\delta_{D_2} * \cdots *\delta_{D_k} *\cdots.
\end{equation*}

It is clear that the uniformly distributed self-similar measures and non-autonomous similar measures may be regarded as special cases of infinite convolutions; see \cite{Falco03,GM22,GM}. Given a sequence $\{(N_k,B_k)\}_{k=1}^\infty $ where $N_k\ge 2$ and $B_k\subset\R$ is finite for all $k\in\N_+$. We write
\begin{equation}\label{def_mun}
	\mu_k =\delta_{{N_1}^{-1}B_1}\ast\delta_{(N_2N_1)^{-1}B_2}\ast\dots\ast\delta_{(N_kN_{k-1}\cdots N_1)^{-1}B_k}.
\end{equation}
If the sequence $\{\mu_k\}_{k=1}^\infty$ converges weakly to a Borel probability measure $\mu$, then we call $\mu$ the \emph{infinite convolution} of $\{(N_k,B_k)\}_{k=1}^\infty,$ denoted by
\begin{equation}\label{infinite-convolution}
	\mu =\delta_{{N_1}^{-1}B_1}\ast\delta_{(N_2N_1)^{-1}B_2}\ast\dots\ast\delta_{(N_kN_{k-1}\cdots N_1)^{-1}B_k} *\cdots.
\end{equation}

For each integer $k\geq 1$, we write
\begin{equation}\label{def_mugn}
	\mu_{>k}=\delta_{({N_{k+1}\cdots N_1})^{-1}B_{k+1}}\ast\delta_{(N_{k+2}\cdots N_1)^{-1}B_{k+2}}\ast\cdots,
\end{equation}
and it is clear that $\mu = \mu_k * \mu_{>k}$. Let
\begin{equation}\label{def_nu_n}
	\nu_{>k} = \mu_{>k}\circ(N_kN_{k-1}\cdots N_1)^{-1}=\delta_{N_{k+1}^{-1} B_{k+1}} * \delta_{(N_{k+2} N_{k+1})^{-1} B_{k+2}} * \cdots,
\end{equation}
which is crucial to investigate  the spectrality of infinite convolutions (see Theorem \ref{thm_A-spectral eigen set}).   
We say the infinite convolution $\mu$ of $\{(N_k,B_k)\}_{k=1}^\infty$ is \emph{non-degenerate} if there exists a subsequence $\{\nu_{>n_j}\}_{j=1}^\infty$ that converges weakly to a measure singular to Lebesgue measure.”

Since different permutations of  $\{(N_k,B_k)\}_{k=1}^\infty $ may give rise to different infinite convolutions in \eqref{infinite-convolution}, by randomly choosing a permutation, the infinite convolution becomes a measurable mapping of the permutation; in other words, it defines a random measure. Recall the definition of random measures. Let $(\Omega, \mathcal{F})$ and $(E,\mathcal{B})$ be measurable spaces. A mapping $M: \Omega \times \mathcal{B} \to [0,+\infty]$ is called a \emph{random measure} on $(E,\mathcal{B})$ if $(\rmnum{1})$  the mapping $\bw \mapsto M(\bw, B)$ is $\mathcal{F}$-measurable for every $B\in\mathcal{B}$; and $(\rmnum{2})$ $B \mapsto M(\bw, B)$ is a measure on $(E,\mathcal{B})$ for all $\bw\in\Omega$. See \cite{Cinlar11} for details.

Consider finitely many pairs $\{(N_k,B_k)\}_{k=1}^m$ where $N_k\ge 2$ and $B_k\subset\R$ is finite. Let $\Omega=\{1,2,\dots, m\}^\infty$ be the symbolic space over the alphabet $\{1, 2, \dots, m\}$. We topologise $\Omega$  using the metric $d(\bu,\bv) = 2^{-|\bu \wedge \bv|}$ for all $\bu, \bv \in \Omega$ to make $\Omega$  into a complete metric space, see \cite{BBT08} for details.  Let $\mathcal{F}$ be the Borel $\sigma$-algebra  on $\Omega$. Then $(\Omega, \mathcal{F})$ is a measurable space. We write $\mathcal{P}(\Omega)$ for the set of all Borel probability measures on $\Omega$.

Let $\bn=\{ n_{k}\}_{k=1}^\f$ be a sequence of positive integers. For each integer $k>0$, we write 
\begin{equation*} \label{def_mubk}
	\mu^{\bn}_{k}(\bw) =\delta_{N_{\omega_{1}}^{-n_{1}}B_{\omega_{1}}} * \delta_{N_{\omega_{1}}^{-n_{1}}N_{\omega_{2}}^{-n_{2}}B_{\omega_{2}}} *  \dots * \delta_{N_{\omega_{1}}^{-n_{1}}N_{\omega_{2}}^{-n_{2}}\cdots N_{\omega_{k}}^{-n_{k}}B_{\omega_{k}}},
\end{equation*}
for all $\bw\in \Omega$. Suppose that for every $\bw\in \Omega$, the sequence $\{\mu_k^{\mathbf{n}}(\bw)\}_{k=1}^\infty$ converges weakly to $\mu^{\mathbf{n}}(\bw)$,  and write
\begin{equation*} \label{def_mubn}
	\mu^{\mathbf{n}}(\bw) = \delta_{N_{\omega_{1}}^{-n_{1}}B_{\omega_{1}}} * \delta_{N_{\omega_{1}}^{-n_{1}}N_{\omega_{2}}^{-n_{2}}B_{\omega_{2}}} * \cdots .
\end{equation*}
We define a mapping $\mu^{\mathbf{n}}: \Omega\times \mathcal{B}(\mathbb{R}) \to [0,+\f]$ by
\begin{equation} \label{def_random_measure}
	\mu^{\mathbf{n}}(\bw, B)= \mu^{\mathbf{n}}(\bw)(B),
\end{equation}
for all $\bw \in \Omega$ and all   $B \in \mathcal{B}(\mathbb{R})$.
For simplicity,  if $\bn=\{ 1 \}_{k=1}^\f$, we write
\begin{equation} \label{def_corresponding_P_of_random_measure}
	\mu(\bw) =		\mu^\bn(\bw) = \delta_{N_{\omega_{1}}^{-1}B_{\omega_{1}}} * \delta_{N_{\omega_{1}}^{-1}N_{\omega_{2}}^{-1}B_{\omega_{2}}} * \cdots,
\end{equation}
for all $\bw\in\Omega$. Since $\mu^{\mathbf{n}}(\bw)$ exists for every $\bw\in\Omega$, the mapping $\mu^{\mathbf{n}}$ is well-defined.  

\begin{theorem} \label{thm_measurable}
Given $\{(N_k,B_k)\}_{k=1}^m$ where  $N_k\ge2$ and $B_k\subset \R $ is finite. Then for every  sequence $\mathbf{n}$ of positive integers,  the mapping $\mu^{\mathbf{n}}$ given by \eqref{def_random_measure} is a random measure.
\end{theorem}

In this paper, we study analytic properties of these random measures.  
Under non‑degeneracy conditions we prove that $\mu^{\mathbf{n}}$ 
is not only a spectral measure, but also admits uncountably many spectra 
(Theorem~\ref{thm_finite_random}).  For the special case of consecutive 
digit sets we completely characterise the positive integer spectral 
eigenvalues (Theorem~\ref{thm_finite_consecutive}).

\subsection{Spectral measures and spectral eigenvalues}
A Borel probability measure $\mu$ on $\R$ is called a \emph{spectral measure} if there exists a countable subset $\Lambda \sse \R$ such that the family of exponential functions
$$
\set{e_\lambda(x) = e^{-2\pi i \lambda\cdot x}: \lambda \in \Lambda}
$$
forms an orthonormal basis in $L^2(\mu)$, and we call the set $\Lambda$   a \emph{spectrum} of $\mu$ and call $(\mu, \Lambda)$ a {\em spectral pair}.
The existence of spectrum of measures is a fundamental question in harmonic analysis, which was first studied by Fuglede~\cite{Fuglede-1974} for the normalised Lebesgue measure on measurable sets. In particular,  Jorgensen and Pedersen~\cite{Jorgensen-Pedersen-1998}  found that some self-similar measures may have spectra, which are a special class of infinite convolutions.

Admissible pairs are the key to study the spectrality of infinite convolutions. Given an integer $N\ge2$ and a finite subset  $B\sse \Z$ with $\# B\ge 2$.  If there exists $L\sse \Z$ such that the matrix
$$
\left[ \frac{1}{\sqrt{\# B}} e^{-2 \pi i  \frac{b\cdot l}{N}}  \right]_{b \in B, l \in L} 
$$
is unitary, we call $(N, B)$ an {\it admissible pair} in $\R$ and call $(N,B,L)$ a {\it Hadamard triple} in $\R$. This implies $\# B=\# L$, and $(N,B,L)$ being a Hadamard triple is equivalent to $L$ being a spectrum of $\delta_{N^{-1}B}$. In \cite{Laba-Wang-2002}, {\L}aba and Wang proved that self-similar measures generated by admissible pairs  with equal weights in $\R$ are  spectral measures, and they found  a special example, both $\Lambda$ and $2\Lambda$ are spectra of a singular continuous spectral measure.  

A  real $t\neq 0$ is called a {\it spectral eigenvalue} of $\mu$ if there exists $\Lambda\subseteq\R$ such that both $\Lambda$ and $t\Lambda$ are spectra of $\mu$. Similarly, given an admissible pair $(N,B)$ associated with the spectral measure $\delta_{N^{-1}B}$,  we say that $t\in\mathbb{R}$ is an \textit{admissible eigenvalue} of $(N,B)$ provided that whenever $(N,B,L)$ forms a Hadamard  triple for some $L\subset\mathbb{Z}$, $(N,B,tL)$ is also a Hadamard triple. The spectral eigenvalue problem has been extensively studied for self-similar measures (see, e.g., \cite{Dai-2016, He-Tang-Wu-2019, Li-Wu-2022, Kong-Li-Wang-2026, Lu-2026}) and significant progress has also been made for self-affine measures in higher dimensions \cite{An-Dong-He-2022, Chen-Liu-2023, Liu-2024, Liu-Tang-Wu-2023}. Fu and He \cite{Fu-He-2018} completely characterised the spectral eigenvalues of infinite convolutions where $\# B_k\equiv2$ or $3$ and the contraction ratio remains constant. Chen, Liu and Liu \cite{Chen-Liu-Liu-2026} obtained a necessary condition for $t$ to be a spectral eigenvalue of the infinite convolution where $B_k$ is a consecutive digit set,  i.e. $B_k = \{0,1,\dots,b_k-1\}$ for all $k\in\N_+$. Moreover, under certain restrictions, they further obtained the corresponding necessary and sufficient condition.

Next, we present a sufficient condition ensuring that an infinite convolution  is a spectral measure and  possesses  spectral eigenvalues. Note that the $t$-equi-positive measure is defined by Definition~\ref{def_equi_measure}.
\begin{theorem}\label{thm_A-spectral eigen set}
Let  $\mu$ be the  infinite convolution of  a sequence  $\{(N_k,B_k)\}_{k=1}^\infty$ of admissible pairs in $\mathbb{R}$.
Let $\mathcal{E}$ be the set of all positive integers $t$ such that $t$ is an admissible eigenvalue of $(N_k,B_k)$ for all $k\in\mathbb{N}_+$ and $\mu$ is a $t$-equi-positive measure.
\begin{enumerate}
\item If $\mathcal{E}\neq\emptyset$, then $\mu$ is a spectral measure.
\item If there exists $t\in\mathcal{E}$ with $t>1$, then there exist uncountably many subsets $\Lambda\subset\mathbb{Z}$ such that $t\Lambda$ is a spectrum of $\mu$ for every $t\in\mathcal{E}$. Moreover, each element of $\mathcal{E}$ is a spectral eigenvalue of $\mu$.
\end{enumerate}
\end{theorem}

If $B_k$ is a consecutive digit set,  i.e. $B_k = \{0,1,\dots,b_k-1\}$ for all $k\in\N_+$,   we obtain a complete characterisation of the positive integer spectral eigenvalues.

\begin{theorem}\label{thm_finite_consecutive}
Given finitely many admissible pairs $\{(N_k, B_k)\}_{k=1}^m$ in $\mathbb{R}$ and a sequence $\bn$ of positive integers. Assume  $B_k = \{0,1,\dots,b_k-1\}$ with $b_k\in\mathbb{N}_+$  for each $1\le k\le m$.
Let $\mu^\mbf{n}$, $\mu$ be given by \eqref{def_random_measure}, \eqref{def_corresponding_P_of_random_measure}, and set
\[
\mathcal{E}_{\bw} = \bigl\{t\in\mathbb{N}_+ : \gcd(t,b_{\omega_k}) = 1 \text{ for all } k\in\mathbb{N}_+\bigr\}.
\] 
Then the following statements hold for every $\bw$ such that $\mu(\bw)$ is non-degenerate:
\begin{enumerate}
\item $\mu^{\mathbf{n}}(\bw)$ is a spectral measure;
\item $\mathcal{E}_{\bw}$ is precisely the set of all positive integer spectral eigenvalues of $\mu^{\mathbf{n}}(\bw)$.
\end{enumerate}
Furthermore, there exist uncountably many subsets $\Lambda_{\bw} \subseteq \mathbb{Z}$ such that $t\Lambda_{\bw}$ is a spectrum of $\mu^{\mathbf{n}}(\bw)$ for every $t\in \mathcal{E}_{\bw}$.
\end{theorem}

\subsection{Spectrality properties of random measures}
In this subsection, we extend the notions of spectrality and spectral eigenvalues from deterministic measures to random measures.

Let  $M$ be a random measure on  $(\mathbb{R},\mathcal{B}(\mathbb{R}))$ with respect to  $(\Omega, \mathcal{F})$. We say $M$ is a \emph{spectral random measure} if $M(\bw)$ is a spectral measure for all $\bw\in\Omega$. And we call $t$ a spectral eigenvalue of $M$ if $t$ is a spectral eigenvalue of $M(\bw)$ for all $\bw\in\Omega$.

Moreover, given $\mathbb{P}\in\mathcal{P}(\Omega)$, we say $M$ is \emph{a spectral random measure for $\mathbb{P}$-a.e. $\omega\in \Omega$  (or  the random measure $M$ is spectral $\mathbb{P}$-a.e.)} if $M(\bw)$ is a spectral measure for $\mathbb{P}$-a.e. $\bw\in\Omega$. $M$ is said to be {\it singular} $\mathbb{P}$-a.e., if $M(\bw)$ is singular to the Lebesgue measure for $\mathbb{P}$-a.e. $\bw\in\Omega$. 
And we call $t$ \emph{a spectral eigenvalue of $M$ $\mathbb{P}$-a.e} if $t$ is a spectral eigenvalue of $M(\bw)$ for $\mathbb{P}$-a.e. $\bw\in\Omega$.

\begin{definition}
	Let $\mathcal{E}$ be a subset of $\mathbb{R}$ containing $1$ and $M$ be a random measure on  $(\mathbb{R},\mathcal{B}(\mathbb{R}))$ with respect to  $(\Omega, \mathcal{F})$. We say $\mathcal{E}$ is an $\aleph_1$-spectral eigen set of $M$, if for all $\bw\in\Omega$, there exist uncountably many subsets $\Lambda_{\bw}\subset\R$ such that $t\Lambda_{\bw}$ is a spectrum of $M(\bw)$ for all $t\in \mathcal{E}$.
	
	Moreover, given $\mathbb{P} \in \mathcal{P}(\Omega)$, we say  $\mathcal{E}$ is an $\aleph_1$-spectral eigen set of $M$ $\mathbb{P}$-a.e. if the above condition holds for $\mathbb{P}$-a.e. $\bw\in\Omega$.
\end{definition}

Given  a sequence $\bn$ of positive integers and finitely many admissible pairs $\{(N_k, B_k)\}_{k=1}^m$ in $\mathbb{R}$ where  $N_k\ge2$ and $B_k\subset \R $ is finite.    Write 
\begin{equation}\label{def_em}
\mathcal{E}_m=\{t\in\N_+|\gcd(t,N_k)=1,\ \text{for all} \ 1 \leq k \leq m\}.
\end{equation}
By Theorem \ref{thm_measurable},  $\mu^\mathbf{n}$ and $\mu$ given by \eqref{def_random_measure} and \eqref{def_corresponding_P_of_random_measure} are random measures. Moreover, we show that they are also spectral random measures and that $\mathcal{E}_m$ is an $\aleph_1$-spectral eigen set of $\mu^\mathbf{n}$.

\begin{theorem}\label{thm_finite_random}
Given   $\mathbb{P}\in\mathcal{P}(\Omega)$. Let $\mu^\mathbf{n}$, $\mu$ be given by \eqref{def_random_measure}, \eqref{def_corresponding_P_of_random_measure} and $\mathcal{E}_m$ be given by \eqref{def_em}.  If $\mu$ is non-degenerate $\mathbb{P}$-a.e., then $\mu^\mathbf{n}$ is a spectral random measure $\mathbb{P}$-a.e., and $\mathcal{E}_m$ is an $\aleph_1$-spectral eigen set of $\mu^\mathbf{n}$ $\mathbb{P}$-a.e.
\end{theorem}

If $\# B_k< N_k$ for all $  k $, it is obvious that $\mu(\bw)$ is non-degenerate for all $\bw\in\Omega$.

\begin{corollary}\label{cor_finite_random_1}
Let $\mu^\mathbf{n}$, $\mu$ be given by \eqref{def_random_measure}, \eqref{def_corresponding_P_of_random_measure} and $\mathcal{E}_m$ be given by \eqref{def_em}.  Suppose that $\# B_k< N_k$ for $1 \leq k \leq m$.  
Then $\mu^\mathbf{n}$ is a spectral random measure, and $\mathcal{E}_m$ is an $\aleph_1$-spectral eigen set of $\mu^\mathbf{n}$.
\end{corollary}

In the consecutive digit case, Corollary~\ref{cor_finite_random_1} remains valid with the larger se	t $\mathcal{E}'_m$ defined below.
\begin{corollary}\label{cor_finite_random_2}
Assume  $B_k = \{0,1,\dots,b_k-1\}$ with $b_k\in\mathbb{N}_+$  for each $1\le k\le m$.
Let $\mu^\mbf{n}$, $\mu$ be given by \eqref{def_random_measure}, \eqref{def_corresponding_P_of_random_measure}, and set
\[
\mathcal{E}'_{m} = \bigl\{t\in\mathbb{N}_+ : \gcd(t,b_{k}) = 1 \text{ for all } 1\le k\le m\bigr\}.
\] 	
If $\mu$ is non-degenerate $\mathbb{P}$-a.e., then $\mu^\mathbf{n}$ is a spectral random measure, and $\mathcal{E}'_m$ is an $\aleph_1$-spectral eigen set of $\mu^\mathbf{n}$ $\mathbb{P}$-a.e.
\end{corollary}

Given a probability vector $\mathbf{p}=(p_{1}, p_{2}, \cdots,p_m)$, i.e., $\sum_{j=1}^{m} p_j=1$ where $p_j\geq 0$ for all $1\le j\le m$,  we define a probability measure $\mathbb{P}$ on $\Omega$ by setting
\begin{equation} \label{def_BPM}
	\mathbb{P}([\bu]) =p_{\bu}\equiv p_{\alpha_{1}}p_{\alpha_{2}}\cdots p_{\alpha_{k}}  \quad (\bu=\alpha_1\cdots \alpha_k)
\end{equation}
for each cylinder $[\bu]$ and
extending to general subsets of $\Omega$ in the usual way. The probability measure $\mathbb{P}$ is called the {\it Bernoulli measure} associated with the probability vector $\mathbf{p}$. A probability vector $\mathbf{p}=(p_{1}, p_{2}, \cdots,p_m)$ is called {\it positive} if $p_{j}>0$ for all $1\le j\le m$.

\begin{theorem}\label{thm_finite_random_Bernoulli}
Let $\mu^\mathbf{n}$, $\mu$ be given by \eqref{def_random_measure}, \eqref{def_corresponding_P_of_random_measure} and $\mathcal{E}_m$ be given by \eqref{def_em}.   Let $\mathbb{P}$ on $\Omega$ be  a Bernoulli probability measure given by \eqref{def_BPM}. 
If $\mu$ is singular to Lebesgue measure $\mathbb{P}$-a.e., then $\mu^\mathbf{n}$ is a spectral random measure $\mathbb{P}$-a.e., and $\mathcal{E}_m$ is an $\aleph_1$-spectral eigen set of $\mu^\mathbf{n}$ $\mathbb{P}$-a.e.
\end{theorem}

The rest of the paper is organised as follows. In Section \ref{sec_preliminaries}, we collect some basic properties of admissible pairs and list several lemmas and theorems used later in the paper. In Section \ref{sec_Random measure}, we prove Theorem \ref{thm_measurable} by establishing the required measurability. In Section \ref{sec_t}, we prove Theorem \ref{thm_A-spectral eigen set} and investigate the spectral eigenvalues of infinite convolutions under  $t$-equi-positivity. Since $t$-equi-positivity is often difficult to verify directly, in Section~\ref{sec_Random spectral measures}, we show that $t$-admissibility implies $t$-equi-positive and then use this implication to prove Theorem~\ref{thm_finite_random} via periodic zero sets. In the final section, we prove Theorem \ref{thm_finite_consecutive}, Corollary \ref{cor_finite_random_1}, Corollary \ref{cor_finite_random_2} and Theorem \ref{thm_finite_random_Bernoulli}.

\section{Preliminaries}\label{sec_preliminaries}
	
First, we list some useful properties of admissible pairs.
	
\begin{lemma}\label{lem_hadamard}\cite{Dutkay-Haussermann-Lai-2019, Laba-Wang-2002}
Suppose that $(N,B,L)$ is a Hadamard triple in $\R$. Then \\	
\noindent	{\rm(i)}$(N,B+b_0,L+l_0)$ is also a Hadamard triple for all $l_0,b_0\in\R$;
		
\noindent	{\rm(ii)} If $B'\equiv B \pmod{N}$ and $L'\equiv L \pmod{N}$, then $(N,B',L')$ is a Hadamard triple;
		
\noindent	{\rm(iii)} Given a finite sequence $\{(N_k,B_k,L_k)\}_{k=1}^n$ of  Hadamard triples in $\R^d$. Write $\mbf N=N_nN_{n-1}\cdots N_1$,
$$
\mbf B=N_nN_{n-1}\cdots N_2 B_1+\cdots+N_nB_{n-1}+B_n,
$$
and
$$
\mbf L=L_1+N_1L_2+\cdots+(N_{n-1}\cdots N_2N_1)L_n.
$$
Then $(\mbf N,\mbf B,\mbf L)$ is a Hadamard triple.
\end{lemma}

\begin{lemma}\label{lem_gcd-Hadamard}\cite{Kong-Li-Wang-2026}
Let $(N,B)$ be an admissible pair and $t$ be a non-zero integer. If $\gcd(t,N)=1$, then $t$ is an admissible eigenvalue of $(N,B)$.
\end{lemma}

\begin{lemma}\label{lem_hadamard'}
Let $(N,B,L)$ be a Hadamard triple. If $\# B< N$, then $\# B\le \frac{N}{2}$.
\end{lemma}
\begin{proof}
By Lemma \ref{lem_hadamard}, we assume   $B,L\subseteq \{0,1,\dots,N-1\}$.
Let $C = \{0,1,2,\dots,N-1\}$ and define $B' = C\setminus B$. Since $(N,B,L)$ is a Hadamard triple, the matrix
\[
\left[ \frac{1}{\sqrt{\#B}} e^{-2\pi i \frac{bl}{N}} \right]_{b\in B,\, l\in L}
\]
is  unitary. It is clear that $(N,C,C)$ is a Hadamard triple, and we have the unitary   matrix
\[
\left[ \frac{1}{\sqrt{\#C}} e^{-2\pi i \frac{bl}{N}} \right]_{b\in C,\, l\in C}. 
\]

Since $\#B < N$ and   $B'\neq \emptyset$, for each pair $l\neq l'\in L$, we have
\begin{equation*}
	\int_{\mathbb{R}} e^{-2\pi i \frac{b(l-l')}{N} x}\,d\delta_{B'}(x)
	=\sum_{b\in B'} e^{-2\pi i \frac{b(l-l')}{N}} =\sum_{b\in C} e^{-2\pi i \frac{b(l-l')}{N}}-\sum_{b\in B} e^{-2\pi i \frac{b(l-l')}{N}}=0.
\end{equation*}
This implies that   $\{e^{-2\pi i l\cdot x} : l\in L\}$ forms an orthonormal set in $L^2(\delta_{B'})$.   Thus $\# L\le \dim L^2(\delta_{B'})$,  that is $\#B = \#L \le \#B'$, and hence   
\[
N = \#B + \#B' \ge 2\#B.
\]
\end{proof}

Let $\mcal{P}(\R)$ denote the set of all Borel probability measures on $\R$, and we write $\mathcal{T}_w$  for the weak topology  on $\mathcal{P}(\mathbb{R})$.
For $\mu \in \mcal{P}(\R)$, the \emph{Fourier transform} of $\mu$ is given by
$$
\wh{\mu}(\xi) = \int_{\R} e^{-2\pi i \xi \cdot x} \D \mu(x).
$$
For a finite set $B\subset \Z$, we write  
\begin{equation}\label{def_FMB}
M_B(\xi)=\wh{\delta_B}(\xi)=\frac{1}{\# B}\sum_{b\in B}e^{-2\pi i b \cdot\xi}.
\end{equation}
	
Given a Borel probability measure $\mu$ and a subset $\Lambda\subseteq\R$, we write
\begin{align*}
Q_{\mu,\Lambda}(\xi)=\sum_{\lambda\in \Lambda}|\wh{\mu}(\xi+\lambda)|^2,
\end{align*}
where $Q_{\mu,\Lambda}(\xi)=0$ for all $\xi\in\R$ if $\Lambda=\emptyset$.
The following theorem is often used to verify the spectrality of measures; see~\cite{Jorgensen-Pedersen-1998, LMW-2023} for the proofs.
	
\begin{theorem}\label{thm_Q} 
Let $\mu$ be a probability measure on $\R$, $\Lambda\subseteq\R$. Then		
		
\noindent (i) the set $\{e^{-2\pi i \lambda\cdot x}:\lambda\in\Lambda\}$ is an  orthonormal set in $L^2(\mu)$ if and only if $Q_{\mu,\Lambda}(\xi)\le1$ for all $\xi\in\R$.
		
\noindent (ii) the set $\{e^{-2\pi i \lambda\cdot x}:\lambda\in\Lambda\}$ is an  orthonormal basis in $L^2(\mu)$ if and only if $Q_{\mu,\Lambda}(\xi)\equiv1$ for all $\xi\in\R$.
\end{theorem}
	
Next,  we recall four classical theorems which are useful for studying weak limits of probability measures in this paper;  see \cite {Bil03} for proofs.  
\begin{theorem}\label{thm_weak converge}
	Let $\mu, \mu_1, \mu_2, \dots \in \mathcal{P}(\mathbb{R}^d)$. Then $\mu_n$ converges weakly to $\mu$ if and only if
	\[
	\lim_{n\to\infty} \hat{\mu}_n(\xi) = \hat{\mu}(\xi) \quad \text{for every } \xi \in \mathbb{R}^d.
	\]
\end{theorem}

\begin{theorem}\label{thm_factorization}
	Let $\{\mu_k\}_{k=1}^\f,\{\nu_k\}_{k=1}^\f\subseteq \mathcal{P}(\R)$. If $\mu_k$ and $\nu_k$ converge weakly to $\mu$ and $\nu$ respectively, then we have $\mu_k*\nu_k$ converges weakly to $\mu*\nu$.
\end{theorem}	

\begin{theorem}\label{thm_tight-subsequence weak converge}
	Let $\Phi \subseteq \mathcal{P}(\mathbb{R}^d)$. Then $\Phi$ is tight if and only if, for every sequence $\{\mu_n\} \subseteq \Phi$, there exists a subsequence $\{\mu_{n_j}\}$ and a Borel probability measure $\mu \in \mathcal{P}(\mathbb{R}^d)$ such that $\mu_{n_j}$ converges weakly to $\mu$ as $j \to \infty$.
\end{theorem}

\begin{theorem} \label{thm_portmanteau}
Let $(X,d)$ be a metric space and let $\mu_n, \mu$ be Borel probability
measures on $X$.  The following are equivalent:
\begin{enumerate}
\item[$(\rmnum{1})$] $\mu_n \to \mu$ weakly, i.e., 
$\int_X f \,\mathrm{d}\mu_n \to \int_X f \,\mathrm{d}\mu$ 
for every bounded continuous function $f : X \to \mathbb{R}$.
\item[$(\rmnum{2})$] $\liminf_{n\to\infty} \mu_n(U) \ge \mu(U)$ 
for every open set $U \subseteq X$.
\item[$(\rmnum{3})$] $\limsup_{n\to\infty} \mu_n(F) \le \mu(F)$ 
for every closed set $F \subseteq X$.
\item[$(\rmnum{4})$] $\lim_{n\to\infty} \mu_n(A) = \mu(A)$ 
for every $\mu$-continuity set $A$ (i.e., $\mu(\partial A) = 0$).
\end{enumerate}
\end{theorem}

Finally,   we give a necessary and sufficient condition for the singularity of infinite convolutions relative to Lebesgue measure, and  this conclusion is frequently used   in our proofs.

\begin{lemma}\label{lem_singular}
Let $\{(N_k,B_k)\}_{k=1}^\infty$ be a sequence of admissible pairs in $\R$ satisfying 
$$
\sup_k\bigg\{\frac{\max\{|b|:b\in B_k\}}{N_k}\bigg\}<\f.
$$
Then the infinite convolution $\mu$ of  $\{(N_k,B_k)\}_{k=1}^\infty $ is singular  to the Lebesgue measure if and only if 
$$
\# \{k\in\N_+:\# B_k< N_k\}=\f.
$$
\end{lemma}
\begin{proof}
First, we claim that $\# \{k\in\N_+:\# B_k< N_k\}=\f$ if and only if $\mathcal{L}(\operatorname{supp} \mu)=0$, where $\mathcal{L}$ denotes the Lebesgue measure. 
	
Let  $d>0$ be an integer such that
$$
\sup_k\bigg\{\frac{\max\{|b|:b\in B_k\}}{N_k}\bigg\}<\frac{d}{2}.
$$
Let $\mu_k$ be given by \eqref{def_mun}. Then we have $\operatorname{supp} \mu_k=\{\sum_{j=1}^{k}\frac{b_j}{N_jN_{j-1}\cdots N_1}:b_j\in B_j\}$, $\operatorname{supp} \mu=\overline{\{\sum_{j=1}^{\f}\frac{b_j}{N_jN_{j-1}\cdots N_1}:b_j\in B_j\}}$, and we write 
$$
\eta_k=\sum_{a\in\operatorname{supp} \mu_k}  \mathcal{L}|_{a+\frac{d}{N_kN_{k-1}\cdots N_1}[-1,1]},
$$
where $\mathcal{L}|_{[a,b]}$ denotes the restriction of the Lebesgue measure to the interval $[a,b]$.  It is clear that $\{\operatorname{supp}\eta_k\}_{k=1}^\f$ is a decreasing sequence of sets and
$\operatorname{supp} \mu=\bigcap_{k=1}^\f\operatorname{supp}\eta_k$ for all $k\in\N_+$. Thus 
$$
\mathcal{L}(\operatorname{supp} \mu)=\lim_{k\to \f}\mathcal{L}(\operatorname{supp} \eta_k).
$$

Since $\{(N_k,B_k)\}_{k=1}^\infty$ is a sequence of admissible pairs, for distinct $b_1b_2\cdots b_k \neq b_1'b_2'\cdots b_k'$ where $b_j, b_j' \in B_{j}$, we have
\[
\sum_{j=1}^{k}\frac{b_j}{N_jN_{j-1}\cdots N_1}\ne\sum_{j=1}^{k}\frac{b'_j}{N_jN_{j-1}\cdots N_1}
\]
Otherwise,
\[
b_k+N_{k}b_{k-1}+\cdots+N_kN_{k-1}\cdots N_1b_1=b'_k+N_{k}b'_{k-1}+\cdots+N_kN_{k-1}\cdots N_1b'_1.
\]
Let $j_0 = \max\{1 \le j \le k : b_j \neq b_j'\}$. Then we have $N_{j_0} \mid b_{j_0} - b_{j_0}'$. However, $b_{j_0}, b_{j_0}' \in B_{j_0}$, which leads to a contradiction. Thus,
\[
\#\operatorname{supp}\mu_k = \#B_k \#B_{k-1}\cdots \#B_1.
\]

For each $k\in\mathbb{N}_+$, let  $Q_k=N_1\cdots N_k$. Since $d\ge1$ and  every  $a\in\operatorname{supp}\mu_k$ may be written as $\frac{m}{Q_k}$ for some $m\in\Z$, the interval $I_a := [a-\frac1{2Q_k},\,a+\frac1{2Q_k}] \subset a+\frac{d}{Q_k}[-1,1]$ has length $1/Q_k$, and these intervals are pairwise disjoint.    Therefore we have
\[
\mathcal{L}(\operatorname{supp}\eta_k)\ge \sum_{a\in\operatorname{supp}\mu_k} \mathcal{L}(I_a)
= \frac{\#\operatorname{supp}\mu_k}{Q_k}
= \frac{\#B_k \#B_{k-1}\cdots \#B_1}{N_k N_{k-1}\cdots N_1}.
\]
Since each interval $a+\frac{d}{Q_k}[-1,1]$ has length $2d/Q_k$,  we have 
\[
\frac{\#B_k \#B_{k-1}\cdots \#B_1}{N_k N_{k-1}\cdots N_1}
\le \mathcal{L}(\operatorname{supp}\eta_k)
\le \frac{2d\,\#B_k \#B_{k-1}\cdots \#B_1}{N_k N_{k-1}\cdots N_1}.
\]

Since $(N_k,B_k)$ is an admissible pair for all $k\in\N_+$, by Lemma \ref{lem_hadamard'}, if $\# B_k< N_k$, we have $\# B_k\le \frac{N_k}{2}$, which implies that $\# \{k\in\N_+:\# B_k< N_k\}=\f$ if and only if $\mathcal{L}(\operatorname{supp} \mu)=0$.

By the claim, the necessity for the conclusion is clear, and we prove   the sufficiency by contradiction. Suppose  $\# \{k\in\N_+:\# B_k< N_k\}<\f$.

For $x\in\operatorname{supp} \mu$, we write $r_k=1/Q_k$  for each $k>0$.
The set $\operatorname{supp}\mu_{>k}$ is contained in $\bigl[-\frac{d}{Q_k},\frac{d}{Q_k}\bigr]$, and  for each $x_0\in\operatorname{supp}\mu_k$, the ``tail set''
$$
\overline{\bigg\{x_0+\sum_{j=k+1}^{\f}\frac{b_j}{N_jN_{j-1}\cdots N_1}:b_j\in B_j\bigg\}}
$$
is a subset of $\bigl[x_0-\frac{d}{Q_k},\,x_0+\frac{d}{Q_k}\bigr]$. Since a tail set intersects $[x-r_k,x+r_k]$ only if its centre $x_0$ satisfies $|x-x_0|\le \frac{1}{Q_k}+\frac{d}{Q_k}= \frac{d+1}{Q_k}$.  Recall that the points of $\operatorname{supp}\mu_k$ are separated by at least $1/Q_k$, and the number of such $x_0$ is at most $ 2d+3$, which implies 
\[
\mu\bigl([x-r_k,x+r_k]\bigr) \le  (2d+3)(\#B_k\#B_{k-1}\cdots\#B_1)^{-1}.
\]
Therefore we have 
\[
\frac{\mu([x-r_k,x+r_k])}{\mathcal{L}([x-r_k,x+r_k])}
\le  \frac{2d+3}{2}\,
\frac{N_{k}N_{k-1}\cdots N_1}{\# B_{k}\# B_{k-1}\cdots\# B_1}.
\]

Since $\# \{k\in\N_+:\# B_k< N_k\}<\f$, for all $x\in\operatorname{supp} \mu$,  we have
$$
\liminf_{r\to 0}\frac{\mu([x-r,x+r])}{\mathcal{L}([x-r,x+r])}\le\liminf_{k\to \f}\frac{\mu([x-r_k,x+r_k])}{\mathcal{L}([x-r_k,x+r_k])}<\f,
$$
and  by Theorem 2.12 in \cite{Mattila-1995}, $\mu$ is absolutely continuous with respect  to $\mathcal{L}$.
\end{proof}

\section{Random measures  generated by infinite convolutions}\label{sec_Random measure}
We recall some definitions and conclusions from probability theory which are used in our proofs. 	Let $X$ be a non-empty set and  $D$  a non-empty collection of some subsets of $X$.
We say $D$ is \emph{a $\lambda$-system on $X$} if it satisfies the following
	\begin{enumerate}
			\item[(\rmnum{1})]$X \in D$;
		\item[(\rmnum{2})] if $A, B\in D$ with $A\subseteq B$, then $B\setminus A \in D$;
				\item[(\rmnum{3})]if $A_{1} \subseteq A_{2} \subseteq A_{3} \subseteq \dots $ is an increasing sequence of sets in $D$, then $\bigcup_{n=1}^{\infty}A_{n} \in D$.
\end{enumerate}
Let $P$ be  a non-empty collection of some subsets of $X$. We say  $P$ is \emph{a $\pi$-system on $X$} if $A \cap B \in P$ for all $A, B\in P$. We write  $\sigma(P) $ for the $\sigma$-algebra generated by $P$. The following conclusion is standard in probability theory, and we refer readers to \cite{Cinlar11} for details.
\begin{theorem}\label{thm_dynkin's}
	If $P$ is a $\pi$-system and $D$ is a $\lambda$-system with $P \subseteq D$, then $\sigma(P) \subseteq D$.
\end{theorem}

Given $\{(N_k,B_k)\}_{k=1}^m$ where  $N_k\ge2$ and $B_k\subset \R $ is finite for all $1\le k\le m$. and a sequence of positive integers $\bn=\{ n_{k}\}_{k=1}^\f$,  recall that the mapping $\mu^{\mathbf{n}}$  is defined respectively by
\begin{equation} \label{Redef_rm}
	\mu^{\mathbf{n}}(\bw, B)= \mu^{\mathbf{n}}(\bw)(B)=\delta_{N_{\omega_{1}}^{-n_{1}}B_{\omega_{1}}} * \delta_{N_{\omega_{1}}^{-n_{1}}N_{\omega_{2}}^{-n_{2}}B_{\omega_{2}}} * \cdots ,
\end{equation}
for all $\bw \in \Omega$ and all Borel sets $B \in \mathcal{B}(\mathbb{R})$. Note that there exists a compact subset $K\subset\R$, independent of $\mbf{n}$, such that $\operatorname{supp} \mu^{\mathbf{n}}(\bw) \subseteq K$ for all $\bw\in\Omega$.

We define a mapping $\phi: (\Omega, d) \to (\mathcal{P}(\mathbb{R}), \mathcal{T}_{w})$ by
\begin{equation} \label{def_phi}
	\phi (\bw) = \mu^{\mathbf{n}}(\bw),
\end{equation}
where $\mu^{\mathbf{n}}(\bw)$ is given by \eqref{Redef_rm}.

Given a bounded Lipschitz function   $f : \mathbb{R} \to \mathbb{R}$,  
\[
\|f\|_\infty = \sup_{x\in\mathbb{R}} |f(x)|,\; \qquad
\mathrm{Lip}(f) = \sup_{x\neq y} \frac{|f(x)-f(y)|}{|x-y|}.
\]
The \emph{bounded Lipschitz metric} on $\mathcal{P}(\mathbb{R})$ is defined by
\[
\rho_{\mathrm{L}}(\mu,\nu) = 
\sup\Bigl\{ \Bigl| \int_{\mathbb{R}} f\,\mathrm{d}\mu - \int_{\mathbb{R}} f\,\mathrm{d}\nu \Bigr|
: \|f\|_\infty \le 1,\; \mathrm{Lip}(f) \le 1 \Bigr\}.
\]
The bounded Lipschitz metric metrizes the weak topology $\mathcal{T}_w$  on $\mathcal{P}(\mathbb{R})$, see \cite{Bogachev07}

\begin{lemma}\label{lem_phicts}
	The $\phi$ given by \eqref{def_phi} is continuous.
\end{lemma}
\begin{proof}
	Since $\mathcal{T}_w$ is metrizable with respect to  the bounded Lipschitz metric 
	$\rho_{\mathrm{L}}$, it suffices to show that   $\rho_{\mathrm{L}}\bigl(\mu^{\mathbf{n}}(\bw^{(m)}), \mu^{\mathbf{n}}(\bw)\bigr) \to 0$ if $\bw^{(m)} \to \bw$    in $\Omega$.
	
	Recall that there exists a constant $M>0$ and a compact set $K\subseteq[-M,M]$ 
	such that $\operatorname{supp}\mu^{\mathbf{n}}(\bw)\subseteq K$ 
	for all $\bw\in\Omega$.  we write
	\begin{equation}
		\mu^{\mathbf{n}}_{>k}(\bw) =\delta_{ N_{\omega_{1}}^{-n_{1}} N_{\omega_{2}}^{-n_{2}} \dots N_{\omega_{k+1}}^{-n_{k+1}} B_{\omega_{k+1}}} * \delta_{N_{\omega_{1}}^{-n_{1}} N_{\omega_{2}}^{-n_{2}} \dots N_{\omega_{k+2}}^{-n_{k+2}} B_{\omega_{k+2}}} *  \dots,
	\end{equation}
	and it is clear that $\mu^{\mathbf{n}}(\bw)= \mu^{\mathbf{n}}_{k}(\bw) * \mu^{\mathbf{n}}_{>k}(\bw)$. 
	Given $k>0$,  for all $|\bw \wedge \bw'|=k$,   $\mu^{\mathbf{n}}_{k}(\bw)=\mu^{\mathbf{n}}_{k}(\bw')$.	Moreover, for every $k$,
	\[
	\operatorname{supp} \mu^{\mathbf{n}}_{>k}(\bw) \subseteq 
	\Bigl[ -C\sum_{j=k+1}^{\infty} 2^{-j}, \; C\sum_{j=k+1}^{\infty} 2^{-j} \Bigr]
	\subseteq [ -2^{-k}C, 2^{-k}C]
	\]
	for some constant $C>0$ independent of $\bw$ and $k$.

	Given a $1$-Lipschitz function $f:\mathbb{R}\to\mathbb{R}$ with 
	$\sup|f|\le 1$. For all $|\bw \wedge \bw'|\geq k$,
	\begin{align*}
		\Bigl| \int f\,\mathrm{d}\mu^{\mathbf{n}}(\bw) 
		- \int f\,\mathrm{d}\mu^{\mathbf{n}}(\bw') \Bigr|
		&=\Bigl| \iint f(x+y)\,\mathrm{d}\mu^{\mathbf{n}}_{k}(\bw)(x)\,
		\mathrm{d}  \mu^{\mathbf{n}}_{>k}(\bw)(y) \\
		&\qquad - \iint f(x+z)\,\mathrm{d}\mu^{\mathbf{n}}_{k}(\bw)(x)\,
		\mathrm{d}\mu^{\mathbf{n}}_{>k}(\bw')(z) \Bigr| \\
		&\le \iint |y-z| \,\mathrm{d}\mu^{\mathbf{n}}_{>k}(\bw)(y)\,
		\mathrm{d}\mu^{\mathbf{n}}_{>k}(\bw')(z) \\
		&\le \operatorname{diam}\bigl(\operatorname{supp}\mu^{\mathbf{n}}_{>k}(\bw)
		\cup \operatorname{supp}\mu^{\mathbf{n}}_{>k}(\bw')\bigr) \\
		&\le 2^{-k+1}C.
	\end{align*}
This implies that $\rho_{\mathrm{L}}(\mu^{\mathbf{n}}(\bw),\mu^{\mathbf{n}}(\bw'))\le 2^{-k+1}C$.
	Since $\bw^{(m)}\to\bw$,   for each $k>0$,  there exists   $M$ such that $|\bw^{(m)}\wedge\bw|\geq k$ for all $m\ge M$,  and we have 
	$\rho_{\mathrm{L}}(\mu^{\mathbf{n}}(\bw^{(m)}),\mu^{\mathbf{n}}(\bw))\le 2^{-k+1}C$  for all $m\ge M$. Hence  $\rho_{\mathrm{L}}\bigl(\mu^{\mathbf{n}}(\bw^{(m)}), \mu^{\mathbf{n}}(\bw)\bigr) \to 0$.
\end{proof}

Let $\mathcal{F} = \mathcal{B}(\Omega)$ be the Borel $\sigma$-algebra on $\Omega$ with respect to the metric $d$ and $\mathcal{B}(\mathcal{T}_{w})$ denotes the Borel $\sigma$-algebra generated  by the weak topology.

\begin{proposition}\label{prop_Borel_measurable_mapping}
For each $B \in \mathcal{B}(\mathbb{R})$, the mapping 
$\mu^{\mathbf{n}}(\cdot,B) : \Omega \to \mathbb{R}$ given by \eqref{Redef_rm} 
is $\mathcal{F}$-measurable.
\end{proposition}%

\begin{proof}
	Write
	\[
	\mathcal{G} = \bigl\{ B \in \mathcal{B}(\mathbb{R}) :
	\bw \mapsto \mu^{\mathbf{n}}(\bw,B)
	\text{ is } \mathcal{F}\text{-measurable} \bigr\} .
	\]
	It is clear that $\mathcal{G}\subset \mathcal{B}(\mathbb{R}) $. To show  $\mu^{\mathbf{n}}(\cdot,B)$  is $\mathcal{F}$-measurable, it is equivalent to  prove  $\mathcal{G} = \mathcal{B}(\mathbb{R})$.
	Since the open sets form a $\pi$-system generating $\mathcal{B}(\mathbb{R})$, by Theorem~\ref{thm_dynkin's}, it is sufficient to   show that  $\mathcal{G}$ is a $\lambda$-system that contains all open subsets of $\mathbb{R}$.
	
	First, we show that  $\mathcal{G}$ contains all open sets.  For an open set $U \subseteq \mathbb{R}$, we define  $\pi_U : \mathcal{P}(\mathbb{R}) \to [0,1]$   by $\pi_U(\eta) = \eta(U)$.
	By  $(ii)$ in Theorem \ref{thm_portmanteau},  we have 
	$\liminf_{n} \eta_n(U) \ge \eta(U)$ if $\eta_n \to \eta$ weakly, that is,    $\pi_U$ is lower semicontinuous.   Hence $\pi_U$ is Borel measurable
	with respect to $\mathcal{B}(\mathcal{T}_{w})$.  Since
	$$
	\mu^{\mathbf{n}}(\bw, U) = \pi_U \circ \phi(\bw), 
	$$
	by Lemma \ref{lem_phicts} and the measurability of   $\pi_U$,
	we obtain that $    \mu^{\mathbf{n}}(\bw, U)$ is $\mathcal{F}$-measurable.  Thus every open $U$ belongs to
	$\mathcal{G}$.
	
	To show  $\mathcal{G}$ is a $\lambda$-system, we   verify the three defining properties: 		$(\textrm{i})$ $\mathbb{R} \in \mathcal{G}$   is clear.
	\begin{enumerate}
		\item[(\rmnum{2})] Let $A, B \in \mathcal{G}$ with $A \subseteq B$.
		Then for every $\bw$,
		\[
		\mu^{\mathbf{n}}(\bw, B\setminus A) =
		\mu^{\mathbf{n}}(\bw, B) - \mu^{\mathbf{n}}(\bw, A),
		\]
and it is  $\mathcal{F}$-measurable.  Therefore
		$B\setminus A \in \mathcal{G}$.
		\item[(\rmnum{3})] Let $\{A_n\}_{n=1}^\infty \subseteq \mathcal{G}$ be
		an increasing sequence and set $A = \bigcup_{n=1}^\infty A_n$.
		By $\sigma$-additivity,
		\[
		\mu^{\mathbf{n}}(\bw, A) =
		\lim_{n\to\infty} \mu^{\mathbf{n}}(\bw, A_n)
		\]
		for every $\bw$, and we have that $\bw \mapsto \mu^{\mathbf{n}}(\bw,A)$
		is $\mathcal{F}$-measurable, which implies  $A \in \mathcal{G}$.
	\end{enumerate}
	Therefore $\mathcal{G}$ is   a $\lambda$-system, and the conclusion holds.
\end{proof}

\begin{proof}[Proof of Theorem \ref{thm_measurable}]
The infinite convolution $\mu^{\mathbf{n}}(\bw)$ exists and constitutes a Borel probability measure for every sequence $\mathbf{n}$ of positive integers. By virtue of Proposition \ref{prop_Borel_measurable_mapping}, the map $\mu^{\mathbf{n}}(B)\colon\Omega\to \mathbb{R}$ defined by \eqref{Redef_rm} is $\mathcal{F}$-measurable for each $B \in \mathcal{B}(\mathbb{R})$, and  the mapping $\mu^{\mathbf{n}} \colon \Omega\times \mathcal{B}(\mathbb{R}) \to [0,+\infty)$ is a random measure.
\end{proof}

\section{Spectral infinite convolutions}\label{sec_t}

The equi-positive family is a key tool to study the spectrality of fractal measures with compact support in \cite{An-Fu-Lai-2019,Dutkay-Haussermann-Lai-2019},  and it was then generalized in~\cite{LMW-2023, MZ25} to the current version which is also applicable to infinite convolutions without compact support in $\R$. By the same argument, it is straightforward to  extend it  into $\R^d$.  In this paper, we define the $t$-equi-positive family, and then use this idea to study the spectral eigenvalues of measures.

\begin{definition}\label{def_equi_measure}
	Given a positive integer $t$. We call $\Phi\subset\mcal{P}(\R)$ a {\it $t$-equi-positive family} if there exist reals $\epsilon_0>0$ and $\delta_0>0$ such that for all $x\in[0,t)$ and $\mu \in \Phi$, there exists $k_{x,\mu }\in \Z$ such that
	\[
	| \hat{\mu}(x+y+tk_{x,\mu })| \ge \epsilon_0,
	\]
	for all $|y|< \delta_0$, where $k_{x,\mu }=0$ for $x=0$.
	
	We call the infinite convolution $\mu$ of $\{(N_k,B_k)\}_{k=1}^\infty $ a {\it $t$-equi-positive measure} if there exists a subsequence $\{\nu_{>n_j}\}_{j=1}^\infty$ (defined by \eqref{def_nu_n}) that forms a $t$-equi-positive family.
\end{definition}
Note that  if $\mu$ is a $t$-equi-positive measure, for arbitrary integer $s$ satisfying $s|t$, $\mu$ is a $s$-equi-positive measure, in particular, $\mu$ is a $1$-equi-positive measure.

\begin{proof}[Proof of Theorem \ref{thm_A-spectral eigen set}]
 First, suppose that $\mathcal{E}$ is an infinite set. The finite case follows by the same argument with the enumeration truncated. Since $\mathcal{E}\subseteq\mathbb{N}_+$, we may list its elements in strictly increasing order as 
$$
\mathcal{E}=\{t_0,t_1,\cdots,t_p,\cdots\},
$$
where $t_i< t_j$ if $i<j$. Then for each $p\in\N$, as $\mu$ is a $t_p$-equi-positive measure there exists an increasing sequence $\{n_{i,{t_p}}\}_{i=1}^\infty\subseteq\N_+$, such that the subsequence $\{\nu_{>n_{i,{t_p}}}\}_{i=1}^\infty$ of  $\{\nu_{>k}\}_{k=1}^\infty$ is a $t_p$-equi-positive family. It follows that, for each $p\in\N_+$, there exist $\epsilon_{t_p}>0$ and $\delta_{t_p}>0$ such that for all $x\in[0,t_p) $ and $i\ge1$, there exists an integer $k_{x,i,{t_p}}\in \Z$ such that
\begin{equation}\label{Ineq_epf}
| \hat{\nu}_{>n_{i,{t_p}}}(x+y+t_pk_{x,i,{t_p} })| \ge \epsilon_{t_p},
\end{equation}
for all $|y|< \delta_{t_p}$, where $k_{x,i,{t_p} }=0$ for $x=0$.

Since $(N_k,B_k)$ is an admissible pair on $\mathbb{R}$,  by Lemma \ref{lem_hadamard}  (i) and (ii),  we  may choose a set $L_k$ with $0\in L_k\subset \{0,1,\dots,N_k-1\}$ such that $(N_k,B_k,L_k)$ forms a Hadamard triple for every $k\in\mathbb{N}_+$.
For integers $m>n\ge 0$, we define
\begin{align*}
\mathbf{N}_{n,m}&=N_mN_{m-1}\cdots N_{n+1}, \\
\mathbf{B}_{n,m}&=\mathbf{N}_{n+1,m} B_{n+1}+\mathbf{N}_{n+2,m} B_{n+2}+\cdots+\mathbf{N}_{m-1,m} B_{m-1}+B_m,  \\
\mathbf{L}_{n,m}&=L_{n+1}+\mathbf{N}_{n,n+1}L_{n+2}+\cdots+\mathbf{N}_{n,m-1}L_m.
\end{align*}
By  Lemma \ref{lem_hadamard} (i) and (iii), $(\mathbf{N}_{n,m},\mathbf{B}_{n,m},\mathbf{L}_{n,m})$ is again a Hadamard triple, so $\mathbf{L}_{n,m}$ is a spectrum for $\delta_{\mathbf{N}_{n,m}^{-1}\mathbf{B}_{n,m}}$.
Since every $t\in\mathcal{E}$ is an admissible eigenvalue for each pair $(N_k,B_k)$, the triple $(\mathbf{N}_{n,m},\mathbf{B}_{n,m},t\mathbf{L}_{n,m})$ is also a Hadamard triple, and $t\mathbf{L}_{n,m}$ serves as a spectrum of $\delta_{\mathbf{N}_{n,m}^{-1}\mathbf{B}_{n,m}}$ for all $t\in\mathcal{E}$.

Next,  we construct a sequence of finite subsets $\Lambda_j\subset\Z$ by induction. 

For  $j=1$, let $m_1=n_{1,t_0}$ and $\Lambda_1=\mbf L_{0,m_1}$. Note that $t\Lambda_1$ is a spectrum of $\mu_{m_1}$ given by  \eqref{def_mun} with $0\in\Lambda_1$ for all $t\in \mathcal{E}$.

For each $j\in\N_+$, there exist unique $p\in\N$ and $q\in\N_+$ such that $p\le q$ and $j=1+2+3+\cdots+q-1+p$, then we write $s_j=t_p$, $r_j=t_q$.

For $j\ge 2$, suppose that $m_{j-1}$ and $\Lambda_{j-1}$ have already been defined, with $0\in\Lambda_{j-1}$, such that $t\Lambda_{j-1}$ is a spectrum of $\mu_{m_{j-1}}$ for all $t\in\mathcal{E}$. Since $N_k\ge2,$ we  choose  $m_j\in \{n_{i,{s_j}}\}_{i=1}^\infty$ such that $m_j>m_{j-1}$ and
\begin{equation}\label{eq_lambda}
|\mbf N_{0,m_j}^{-1}\lambda|<\frac{\delta_{s_j}}{2r_j} \qquad  \textrm{for all $\lambda\in\Lambda_{j-1}$}.
\end{equation}
By \eqref{Ineq_epf}, for each $\lambda\in \mbf L_{m_{j-1},m_j}$, we may choose a integer $k_{\lambda,j}\in\Z$ such that
\begin{equation}\label{eq_equi-positive}
| \hat{\nu}_{>m_j}(\mbf N_{m_{j-1},m_j}^{-1}\cdot s_j\lambda+y+s_jk_{\lambda,j })| \ge \epsilon_{s_j},
\end{equation}
for all $|y|<\delta_{s_j}$, where $k_{\lambda,j }=0$ if $\lambda=0$, and we   write
\begin{equation}\label{def_Lmdj}
\Lambda_j=\Lambda_{j-1}+\mbf N_{0,m_{j-1}}\{\lambda+\mbf N_{m_{j-1},m_j}k_{\lambda,j}:\lambda\in\mbf L_{m_{j-1},m_j}\}.
\end{equation}
Since $t\Lambda_{j-1}$ is a spectrum of $\mu_{m_{j-1}}$ for all $t\in \mathcal{E}$, by Lemma \ref{lem_hadamard}(ii) and (iii), we have that $t\Lambda_{j}$ is a spectrum of $\mu_{m_j}$ for all $t\in \mathcal{E}$. Note that $0\in\mbf L_{m_{j-1},m_j}$, $k_{0,j}=0$ and $0\in\Lambda_{j-1}$, and it is clear that  $0\in\Lambda_j$ and $\Lambda_{j-1}\subseteq \Lambda_j$. 

Hence the sequence $\{\Lambda_j\}_{j=1}^\infty$ is defined, and we write
\begin{equation}\label{def-Lambda}
\Lambda=\bigcup_{j=1}^\f\Lambda_j.
\end{equation}

To verify that $t_p\Lambda$ is a spectrum of $\mu$ for every $p\in\mathbb{N}$, by Theorem \ref{thm_Q}, it suffices to prove that for all $\xi\in\mathbb{R}$,
\[
Q_{\mu,t_p\Lambda}(\xi)=\sum_{\lambda\in \Lambda}\bigl|\hat{\mu}(\xi+t_p\lambda)\bigr|^2=1.
\]

Fix $p\in\N$. For all $j\geq 1$, since $t_p\Lambda_{j}$ is a spectrum of $\mu_{m_j}$, by Theorem \ref{thm_Q}, we have that
\begin{equation}\label{eq_Q_j}
Q_{\mu_{m_j},t_p\Lambda_j}(\xi)=\sum_{\lambda\in \Lambda_j}|\hat{\mu}_{m_j}(\xi+t_p\lambda)|^2=1,
\end{equation}
for all $\xi\in\R$, and it implies that
$$
\sum_{\lambda\in \Lambda_j}|\hat{\mu}(\xi+t_p\lambda)|^2 =\sum_{\lambda\in \Lambda_j}|\hat{\mu}_{m_j}(\xi+t_p\lambda)|^2|\hat{\mu}_{>m_j}(\xi+t_p\lambda)|^2
\le\sum_{\lambda\in \Lambda_j}|\hat{\mu}_{m_j}(\xi+t_p\lambda)|^2   =1.
$$
It immediately follows that
\begin{equation}\label{eq_Q}
Q_{\mu,t_p\Lambda}(\xi)\le1 \qquad  \textrm {for all $\xi\in\R$.}
\end{equation}

Fix $\xi\in\R$. We define $f(\lambda)=|\hat{\mu}(\xi+t_p\lambda)|^2 $ for $ \ \lambda\in\Lambda$ and for each $j\in\N_+$,
$$
f_j(\lambda)=\left\{
\begin{array}{rl}
|\hat{\mu}_{m_j}(\xi+t_p\lambda)|^2, & \text{if}\    \lambda\in\Lambda_j,\\
0\ \ \ \ \ \ \ \ \ \ \ \ \ \ \, & \text{if} \ \lambda\in\Lambda\setminus\Lambda_j.
\end{array}\right.
$$
By \eqref{def_mun} and \eqref{def_mugn},  we immediately have 
\begin{equation}\label{def_f_i}
f(\lambda) =|\hat{\mu}_{m_j}(\xi+t_p\lambda)|^2|\hat{\mu}_{>m_j}(\xi+t_p\lambda)|^2 \geq  f_j(\lambda) \big|\hat{\nu}_{>m_j}\big(\mbf N_{0,m_j}^{-1}(\xi+t_p\lambda)\big)\big|^2.
\end{equation}
Note that  $\Lambda_{j-1}\subseteq \Lambda_j$, and it is clear that for each $\lambda\in\Lambda$, there exists $j_\lambda\ge1$ such that $\lambda\in\Lambda_j$ for all $j\ge j_\lambda$. Since $\mu$ is the weak limit of $\mu_j$,  we have $f(\lambda)=\lim_{j\to\f}f_j(\lambda)$.
		
Since $N_k\ge2,$  there exists an integer $j_0\ge1$   such that for $j\ge j_0$
\begin{equation}\label{eq_xi}
|\mbf N_{0,m_j}^{-1}\xi|<\frac{\delta_{t_p}}{2} \ \text{and} \ r_j>t_p.
\end{equation}
Fix $j\ge j_0$. For each $\lambda\in\Lambda_j$, by \eqref{def_Lmdj}, there exist  $\lambda_1\in\Lambda_{j-1}$, $\lambda_2\in\mbf L_{m_{j-1},m_j}$ and $k_{\lambda_2,j}\in \Z $ such that
$$
\lambda=\lambda_1+\mbf N_{0,m_{j-1}}\lambda_2+\mbf N_{0,m_j} k_{\lambda_2,j},
$$
and by  \eqref{eq_lambda} and \eqref{eq_xi}, we have 
$$
|\mbf N_{0,m_j}^{-1}(t_p\lambda_1+\xi)|<\frac{\delta_{t_p}}{2}+\frac{\delta_{s_j}}{2}.
$$
Since there exists $j_1\in\N_+$ such that $\eta_j=1+2+\cdot+(p-2+j+j_1)+p>j_0$ for all $j\in\N_+$,  thus $s_{\eta_j}=t_p$. Then, combining this with \eqref{eq_equi-positive} and \eqref{def_f_i}, we have  that
$$
f(\lambda) \geq f_{\eta_j}(\lambda) \big|\hat{\nu}_{>m_{\eta_j}}\big(\mbf N_{m_{{\eta_j}-1},m_{\eta_j}}^{-1}\cdot t_p\lambda_2+\mbf N_{0,m_{\eta_j}}^{-1}(t_p\lambda_1+\xi))+k_{\lambda_2,\eta_j}\big)\big|^2 \ge\epsilon_{t_p}^2f_{\eta_j}(\lambda).
$$
Therefore, for each $j\in\N_+$, we obtain that $f_{\eta_j}(\lambda)\le\epsilon_{t_p}^{-2}f(\lambda) $ for all $\lambda\in\Lambda$.
		
Let $\rho$ be the counting measure on  $\Gamma$. By \eqref{eq_Q}, it is clear that  $f(\lambda)$ is  $\rho$-integrable, and
$$
Q_{\mu,t_p\Lambda}(\xi)=\sum_{\lambda\in \Lambda}|\hat{\mu}(\xi+t_p\lambda)|^2=\int_{\Lambda} f(\lambda) d\rho(\lambda).
$$
Since $f_{\eta_j} \le\epsilon_{t_p}^{-2} f$ for all $j\in\N_+$,   by the dominated convergence theorem, \eqref{eq_Q_j} and \eqref{def_f_i}, it immediately follows that
$$
Q_{\mu,t_p\Lambda}(\xi) =\lim_{j\to\f}\int_{\Lambda} f_{\eta_j}(\lambda) d\rho(\lambda)=\lim_{j\to\f}\sum_{\lambda\in \Lambda_{\eta_j}}|\hat{\mu}_{m_{\eta_j}}(\xi+t_p\lambda)|^2 =1.
$$
By Theorem \ref{thm_Q}, the infinite convolution $\mu$ is a spectral measure with spectrum $t_p\Lambda\subseteq\Z$. 

Hence for every $t\in \mathcal{E}$,  $t\Lambda$ is a spectrum of $\mu$. 

Since $\mathcal{E}\neq \emptyset$, by Definition \ref{def_equi_measure}, it implies that $1\in \mathcal{E}$. It follows that all elements of $\mathcal{E}$ are spectral eigenvalues of $\mu$.

If there exists some $t\in\mathcal{E}$ with $t>1$, then $t$ is a spectral eigenvalue of $\mu$. We claim that $\# \{k\in\mathbb{N}_+:\# B_k< N_k\}=\f.$ Suppose, to the contrary, that $\#\{k:\# B_k< N_k\}<\infty$. Then there exists $M>0$ such that $\# B_k=N_k$ for all $k\ge M$, and for the infinite convolution $\mu'$ of $\{(N_k,B_k)\}_{k=M}^\f$, by the definitions of   eigenvalues and $t$-equi‑positive measures, we  repeat the previous proof to obtain a new $\Lambda$, such that $t\Lambda$ is a spectrum of $\mu'$ for all $t\in\mathcal{E}$. For simplicity, we assume that $\# B_k=N_k$ for all $k\in\mathbb{N}_+$. Let
\[
L_k=\{0,1,\dots,N_k-1\}.
\]
Then $(N_k,B_k,L_k)$ is a Hadamard triple. Define
\[
\Lambda'=L_1+N_1L_2+\cdots+N_1N_2\cdots N_{k-1}L_k+\cdots.
\]
 By the same argument as before, we conclude that $Q_{\mu,\Lambda'}(\xi)\le1$ for all $\xi\in\mathbb{R}$. By Theorem \ref{thm_Q}, the set
\[
\bigl\{e_\lambda(x) = e^{-2\pi i \lambda x} : \lambda \in \Lambda'\bigr\}
\]
is an orthonormal set in $L^2(\mu)$.  It is clear that $\Lambda'=\N$, and we  have that $\bigl\{e_\lambda(x) = e^{-2\pi i \lambda x} : \lambda \in \Z\bigr\}$ is also an orthonormal set in $L^2(\mu)$.

  However, the spectra $\Lambda$ and $t\Lambda$ constructed above satisfy $\Lambda,t\Lambda\subset\Z$. This implies that an orthonormal basis is strictly contained in an orthonormal set, which gives a contradiction.

In the end, we   show that at infinitely many levels of the construction,  we have at least two inequivalent choices for the block $\mathbf{L}_{m_{j-1},m_j}$, which leads to uncountably
many distinct spectra $\Lambda$.

Consider two Hadamard triples $(N_k,B_k,L_k)$ with $0\in L_k\subset\{0,1,\dots,N_k-1\}$ for $k=1,2$. If $\#L_k<N_k$ for both $k=1,2$, we may translate $L_k$ and then reduce modulo $N_k$ to obtain $L_k'$ such that $0\in L_k'\subset\{0,1,\dots,N_k-2\}$, and $(N_k,B_k,L_k')$ remains a Hadamard triple.

Pick an arbitrary nonzero element $l\in L_1'$, and define $L_1''=\big(L_1'\setminus \{l\}\big)\cup\{l+N_1\}$. Then both $(N_1N_2,\,N_2B_1+B_2,\,L_1'+N_1L_2')$ and $(N_1N_2,\,N_2B_1+B_2,\,L_1''+N_1L_2')$ are Hadamard triples, and $L_1'+N_1L_2'$, $L_1''+N_1L_2'$ are distinct subsets of $\{0,1,\dots,N_1N_2-1\}$.

Since $\#\{k\in\N_+:\#B_k<N_k\}=\f$, we may group each maximal run
of indices with $\#B_j=N_j$ together with the next index for which
$\#B_j<N_j$. By Lemma~\ref{lem_hadamard}(iii), each such grouped
block is again an admissible pair, and its digit set has cardinality
strictly smaller than its scaling factor.  Therefore, without loss of generality, we may assume that $\#B_k<N_k$ for all $k\in\N_+$.

For each $k\ge1$, compose the $(2k-1)$-th and $2k$-th admissible pairs
using Lemma~\ref{lem_hadamard}(iii). By the construction above, each
resulting composed pair admits at least two admissible choices for the
spectrum set $L$. Choosing one of these two options independently for every
$k$ produces a different limit $\Lambda=\bigcup_j\Lambda_j$  
for each infinite binary sequence of choices. Hence there are uncountably
many subsets $\Lambda\subset\mathbb{Z}$ such that $t\Lambda$ is a spectrum
of $\mu$ for every $t\in\mathcal{E}$.
\end{proof}

\section{Random spectral measures}\label{sec_Random spectral measures}
The equi-positivity condition is highly technical. To overcome it, An, Fu and Lai \cite{An-Fu-Lai-2019} come up with admissible families which guarantee the existence of equi-positive families. Inspired by this idea, we introduce  $t$-admissible families to study the spectrality of random measures.

 Given $\mu \in \mathcal{P}(\mathbb{R})$ and $t\in\mathbb{N}_+$, we define the {\it $t$-integral periodic zero set} of $\mu$ by
\begin{equation*}\label{def_t-integral periodic zero set}
	\mathcal{Z}_t(\mu) = \{\xi \in [0,t) : \hat{\mu}(\xi + tk) = 0 \text{ for all } k \in \mathbb{Z}\}.
\end{equation*}

\begin{definition}\label{def_taf}
Given  $\Phi \subseteq \mathcal{P}(\mathbb{R})$ and $t\in\mathbb{N}_+$, we say that $\Phi$ is a \textit{$t$-admissible family} if $\mathcal{Z}_{t}(\mu) = \emptyset$ holds for all $\mu \in \operatorname{cl}(\Phi)$,   where $\operatorname{cl}(\Phi)$ is the closure of $\Phi$ with respect to the weak topology.
\end{definition}

The next lemma states that the Fourier transforms of a tight family of probability measures are  equicontinuous; see \cite{Bil03}.
\begin{lemma}\label{lem_tight-equicontinuous}
Let $\Phi \subseteq \mathcal{P}(\mathbb{R})$. If $\Phi$ is tight, then the collection $\{\hat{\mu} : \mu \in \Phi\}$ is equicontinuous.
\end{lemma}

The following theorem illustrates that $t$-admissibility guarantees $t$-equi-positivity under the tightness condition. 
\begin{theorem}\label{thm_admissible-equi-positive}
Assume $\Phi \subseteq \mathcal{P}(\mathbb{R})$ is tight. If $\Phi$ is a $t$-admissible family, then $\Phi$ is  a $t$-equi-positive family.
\end{theorem}
	
\begin{proof}
First, we claim  for every point $x\in[0,t]$, there exists a constant $\varepsilon_x>0$ such that
\[
\sup\big\{ |\hat{\mu}(x + tk)| : k \in \mathbb{Z} \big\} > \varepsilon_x
\quad\text{holds uniformly for all } \mu \in \Phi.
\]
		
Suppose  there exists $x_0\in[0,t]$ such that for each integer $n\geq 1$, we choose a measure $\mu_n\in\Phi$ such that 
\[
\sup\big\{ |\hat{\mu}_n(x_0 +tk)| : k \in \mathbb{Z} \big\} \le \frac{1}{n}.
\]
Since  $\{\mu_n\}\subset\Phi$ and $\Phi$ is tight, by Theorem \ref{thm_tight-subsequence weak converge}, there exists a subsequence $\{\mu_{n_j}\}$ which converges weakly to some $\mu \in \mathcal{P}(\mathbb{R})$. By  Theorem \ref{thm_weak converge}, we obtain
		\[
		\hat{\mu}(x_0 + tk) = \lim_{j \to \infty} \hat{\mu}_{n_j}(x_0 + tk) = 0
		\quad\text{for all }k\in\mathbb{Z}.
		\]
		This yields $x_0\in\mathcal{Z}_t(\mu)$,  contradicting the fact that $\mathcal{Z}_t(\mu)=\emptyset$ since  $\Phi$ is $t$-admissible and $\mu\in\operatorname{cl}(\Phi)$. Hence the claim holds.
		
By the claim, for each $x\in[0,t]$, there exists a constant $\varepsilon_x>0$ such that for each $\mu\in\Phi$, we  choose   $k_{x,\mu}\in\mathbb{Z}$ satisfying
\begin{equation}\label{eq_epsilon}
|\hat{\mu}(x + tk_{x,\mu})| > \varepsilon_x.
\end{equation}

By Lemma \ref{lem_tight-equicontinuous}, the family $\{\hat{\mu}:\mu\in\Phi\}$ is equicontinuous, and for  each $\varepsilon_x>0$, there exists $\delta_x>0$ such that
		\[
		|\hat{\mu}(y_1)-\hat{\mu}(y_2)| < \frac{\varepsilon_x}{2}
		\]
		for all $\mu\in\Phi$ and   $y_1,y_2\in\mathbb{R}$ satisfying $|y_1-y_2|<\delta_x$. For each $\mu\in\Phi$ and $|y|<\delta_x$, we have
		\begin{equation}\label{eq_admissible-equi-positive_1}
			\begin{aligned}
				|\hat{\mu}(x + y + tk_{x,\mu})|
				&\ge \big|\hat{\mu}(x + tk_{x,\mu})\big| - \big|\hat{\mu}(x + y + tk_{x,\mu}) - \hat{\mu}(x + tk_{x,\mu})\big| \\
				&> \varepsilon_x - \frac{\varepsilon_x}{2} = \frac{\varepsilon_x}{2}.
			\end{aligned}
		\end{equation}
		
Since the collection $\{U(x,\delta_x/2)\}_{x\in[0,t]}$ is  an open cover of $[0,t]$, by the Heine--Borel  theorem, there exist finitely many points $x_1<x_2<\dots<x_p\in[0,t]$ such that  
		\[
		[0,t] \subseteq \bigcup_{j=1}^p U\big(x_j, \tfrac{\delta_{x_j}}{2}\big).
		\]
Let
		\[
		\varepsilon = \min\Big\{ \frac{\varepsilon_{x_j}}{2} : 1\le j\le p \Big\},
		\qquad
		\delta = \min\Big\{ \frac{\delta_{x_j}}{2} : 1\le j\le p \Big\}.
		\]
For each $x\in(0,t]$, there exists  $j\in\{1,2,\dots,p\}$ such that $x\in U(x_j,\delta_{x_j}/2)$. For each $\mu\in\Phi$ and $|y|<\delta$, we have $|x-x_j+y|<\delta_{x_j}$. Combining this with  \eqref{eq_admissible-equi-positive_1}, it follows that 
		\[
		|\hat{\mu}\big(x + y + tk_{x_j,\mu}\big)|
		= \Big|\hat{\mu}\big(x_j + (x-x_j+y) + tk_{x_j,\mu}\big)\Big|
		\ge \frac{\varepsilon_{x_j}}{2} \ge \varepsilon.
		\]
		We   set $k_{x,\mu}:=k_{x_j,\mu}$ for  $x\in(0,t]$. 
		
To define $k_{0,\mu}$ for $x=0$, given $|y|<\delta$,  we have
$|y|<\delta_{x_1}$. From \eqref{eq_epsilon} we have $\varepsilon_{x_1}\le 1$ since
$|\widehat{\mu}(\xi)|\le 1$ for all $\xi$.  Since  $\varepsilon \le \frac{\varepsilon_{x_1}}{2}$, $\widehat{\mu}(0)=1$ and 
$\{\widehat{\mu} : \mu\in\Phi\}$ is equicontinuous, we obtain
		\[
		|\hat{\mu}(y)|
		\ge |\hat{\mu}(0)| - |\hat{\mu}(0)-\hat{\mu}(y)|
		\ge 1 - \frac{\varepsilon_{x_1}}{2} \ge \varepsilon.
		\]
Choosing $k_{0,\mu}:=0$ completes the proof.  
	\end{proof}

Given  admissible pairs $\{(N_k, B_k)\}_{k=1}^m$ in $\mathbb{R}$ and a sequence $\bn$ of bounded positive integers. Recall  that  
\begin{equation} \label{def_mubn2}
	\mu^{\mathbf{n}}(\bw) = \delta_{N_{\omega_{1}}^{-n_{1}}B_{\omega_{1}}} * \delta_{N_{\omega_{1}}^{-n_{1}}N_{\omega_{2}}^{-n_{2}}B_{\omega_{2}}} * \cdots .
\end{equation}
and 
\begin{equation} \label{def_rm2}
	\mu(\bw) =		\mu^\bn(\bw) = \delta_{N_{\omega_{1}}^{-1}B_{\omega_{1}}} * \delta_{N_{\omega_{1}}^{-1}N_{\omega_{2}}^{-1}B_{\omega_{2}}} * \cdots,
\end{equation}
are random measures.    Recall that 
\begin{equation} \label{def_em2}
\mathcal{E}_m=\{t\in\N_+|\gcd(t,N_k)=1,\ \text{for all} \ 1 \leq k \leq m\}.
\end{equation}

Given $f:\mathbb{R}\to\mathbb{C}$, we define the zero set of $f$ as
\[
\mathcal{O}(f) = \{ x \in \mathbb{R} : f(x) = 0 \}.
\]
The following   properties  of zero sets play a key role in our proof. 

\begin{lemma}\label{lem_zero-set-local-finite}
Let $\mu$ be given by \eqref{def_rm2}.  For every $h > 0$, we have 
\[
\#\Big([-h, h] \cap \Big( \bigcup_{\bw \in \Omega} \mathcal{O}(\widehat{\mu(\bw)}) \Big) \Big)<\infty. 
\]
\end{lemma}
\begin{proof} 
For each $j\in\{1,\dots,m\}$, since $B_j\subset\mathbb Z$, by \eqref{def_FMB}, we have that $M_{B_j}$ is $1$-periodic, i.e., $M_{B_j}(\xi+1)=M_{B_j}(\xi)$. Since  $M_{B_j}(0)=1$, we have the zero set
\[
Z_j:=\mathcal{O}(M_{B_j})
\]
is a discrete $1$-periodic subset of $\mathbb R$ with $0\notin Z_j$, that is $Z_j+1=Z_j$. Let
\[
a = \min\bigl\{|z| : z\in Z_1\cup\cdots\cup Z_m \bigr\}>0 , \quad \delta = \frac{1}{2}\min\bigl\{|z_1-z_2| : z_1\neq z_2\in Z_1\cup\cdots\cup Z_m \bigr\}>0 .
\]

For each $\bw=(\omega_k)_{k=1}^{\infty}\in\Omega$,  we have that 
\[
\widehat{\mu(\bw)}(\xi)=\prod_{k=1}^{\infty}
M_{B_{\omega_k}}\!\Bigl(\frac{\xi}{N_{\omega_1}N_{\omega_2}\cdots N_{\omega_k}}\Bigr).
\]
If $\widehat{\mu(\bw)}(\xi)=0$, then  
\[
\mathcal{O}(\widehat{\mu(\bw)}) \subseteq \bigcup_{k=1}^{\infty}
N_{\omega_1}N_{\omega_2}\cdots N_{\omega_k}\, Z_{\omega_k}.
\]
and we have 
\begin{equation}\label{UOW}
\bigcup_{\bw\in\Omega}\mathcal{O}(\widehat{\mu(\bw)})
\subseteq \bigcup_{j=1}^{m}\;
\bigcup_{k_1,\dots,k_m\ge 0}\;
N_1^{k_1}N_2^{k_2}\cdots N_m^{k_m}\, Z_j . 
\end{equation}

Fix $h>0$.  For every $x\in [-h, h] \cap ( \bigcup_{\bw \in \Omega} \mathcal{O}(\widehat{\mu(\bw)}) )$, we write  $x = N_1^{k_1}\cdots N_m^{k_m} z$ for some
$j\in\{1,\dots,m\}$, some  integer $k_1,\dots,k_m$ and
some $z\in Z_j$.  Since $|z|\ge a$ and each $N_i\ge 2$, we obtain
\begin{equation}\label{k1_km}
2^{\,k_1+\cdots+k_m}\le N_1^{k_1}\cdots N_m^{k_m}
= \frac{|x|}{|z|}\le \frac{h}{a}.
\end{equation}
Hence the total exponent $k_1+\cdots+k_m$ is bounded by $\log_2(h/a)$, and it implies that there exist only  finitely many $m$-tuples $(k_1,\dots,k_m)$ satisfying \eqref{k1_km}.  For each such  $m$-tuple $(k_1,\dots,k_m)$, the set $N_1^{k_1}\cdots N_m^{k_m}\, Z_j$ is discrete and
its elements are separated by at least $ N_1^{k_1}\cdots N_m^{k_m}  \delta,
$
so  the set $[-h,h]\cap N_1^{k_1}\cdots N_m^{k_m}\, Z_j$ is finite.
By \eqref{UOW}, the conclusion holds.
\end{proof}

\begin{proposition}\label{pro_finite_zero set}
 Let $\mu$ and $\mathcal{E}_m$ be given by \eqref{def_rm2} and  \eqref{def_em2} respectively. 
For every $\bw$ such that $\mu(\bw)$ is singular to the Lebesgue measure, we have $\mathcal{Z}_t\big(\mu(\bw)\big)=\emptyset$ for all $t\in \mathcal{E}_m$.
\end{proposition}

\begin{proof}	
Since $\widehat{\mu * \delta_a}(\xi) = e^{-2\pi i a \xi} \widehat{\mu}(\xi)$ for each $a\in\mathbb{R}$, it is clear that  $\mathcal{Z}_t(\mu) = \mathcal{Z}_t(\mu * \delta_a)$, which means that the $t$-integral periodic zero set is invariant under translation. Therefore, we assume $0 \in B_k$ for all $1 \leq k \leq m$. By Lemma \ref{lem_hadamard} (i) and (ii), for each $1\le j\le m$, there is $L_j \subseteq \{0,1,\dots,N_j-1\}$ containing $0$ such that $(N_j,B_j,L_j)$ is a Hadamard triple.
	
By Lemma \ref{lem_singular}, it suffices to prove that for every $\bw \in \Omega$ with $\#\{k\in\mathbb{N}_+:\# B_{\omega_k}< N_{\omega_k}\}=\f$, we have $\mathcal{Z}_t(\mu(\bw))=\emptyset$ for all $t\in \mathcal{E}_m$. We prove it by contradiction. 

 Assume that there exist $t\in \mathcal{E}_m$ and $\bw \in \Omega$ with $\#\{k\in\N_+:\# B_{\omega_k}< N_{\omega_k}\}=\infty$ such that $\mathcal{Z}_t\big(\mu(\bw)\big)\ne\emptyset$. For each $j$ and $\ell\in\{0,1,\dots,N_j-1\}$, define
\[
\psi_{\ell,j}(x) = N_j^{-1}(x + t\ell).
\]
Fix  $\xi_0 \in \mathcal{Z}_t\big(\mu(\bw)\big)$ and set $X_0 = \{\xi_0\}$. For $n\ge 1$, we inductively define
\begin{equation}\label{def_Xn}
X_n = \left\{ \psi_{\ell,\omega_n}(\xi) : \xi \in X_{n-1},\ \ell \in \{0,1,\dots,N_{\omega_n}-1\},\ M_{B_{\omega_n}}(\psi_{\ell,\omega_n}(\xi)) \neq 0 \right\}. 
\end{equation}

Since $t\in\mathcal{E}_m$, we have $\gcd(t,N_k)=1$ for each $k$. By Lemma \ref{lem_gcd-Hadamard}, $t$ is an admissible eigenvalue for every $(N_k,B_k)$, so $(N_k,B_k,tL_k)$ is a Hadamard triple for all $k$. By Theorem \ref{thm_Q} (ii), we obtain
\[
\sum_{\ell \in L_{\omega_n}} |M_{B_{\omega_n}}(\psi_{\ell,\omega_n}(\xi))|^2 = 1
\]
for all $\xi\in \R$. In particular, for each $\xi\in X_{n-1}$, there exists at least one element $\ell\in L_{\omega_n}$ such that $M_{B_{\omega_n}}(\psi_{\ell,\omega_n}(\xi))\neq 0$.
	
First, note that distinct finite sequences of indices produce distinct points under iteration. Indeed, suppose two different sequences $(\ell_1,\dots,\ell_n)\neq (\ell_1',\dots,\ell_n')$ yield the same point. Then 
\[
\ell_1 + N_{\omega_1}\ell_2 + \cdots + N_{\omega_1}\cdots N_{\omega_{n-1}}\ell_n 
= \ell_1' + N_{\omega_1}\ell_2' + \cdots + N_{\omega_1}\cdots N_{\omega_{n-1}}\ell_n'.
\]
Let $j_0=\min\{j:\ell_j\neq\ell_j'\}$. We have  $N_{\omega_{j_0}}\mid (\ell_{j_0}-\ell_{j_0}')$, which contradicts the fact that $\ell_{j_0},\ell_{j_0}'\in\{0,1,\dots,N_{\omega_{j_0}}-1\}$.   Hence the map from sequences to points is injective, and we obtain the inequality
\[
\# X_{n-1} \le \# X_n \qquad\text{for all } n\ge 1.
\]

Next, we prove  $X_n \subseteq \mathcal{Z}_t\big(\mu\big(\sigma^n(\bw)\big)\big)$ for all $n\ge 0$ by induction, where $\sigma$ stands for the left shift operator on  $\Omega$, i.e. $\sigma(\bw)=\omega_2\omega_3\cdots$. The case $n=0$ holds trivially by the choice of $\xi_0$. Assume  $X_{n-1} \subseteq \mathcal{Z}_t\big(\mu\big(\sigma^{n-1}(\bw)\big)\big)$ holds. By \eqref{def_Xn}, for  every   $\psi_{\ell,\omega_n}(\xi) \in X_n$ where $\xi\in X_{n-1}$ and $M_{B_{\omega_n}}\big(\psi_{\ell,\omega_n}(\xi)\big)\neq 0$, it is clear that  $\widehat{\mu\big(\sigma^{n-1}(\bw)\big)}(\xi+tk)=0$ for all $k\in\mathbb{Z}$,  and  we obtain
\[
\begin{aligned}
0 &= \widehat{\mu\big(\sigma^{n-1}(\bw)\big)}(\xi + t\ell + tN_{\omega_n}k) \\
&= M_{B_{\omega_n}}\big(\psi_{\ell,\omega_n}(\xi)+tk\big) \widehat{\mu\big(\sigma^{n}(\bw)\big)}\big(\psi_{\ell,\omega_n}(\xi)+tk\big) \\
&= M_{B_{\omega_n}}\big(\psi_{\ell,\omega_n}(\xi)\big) \widehat{\mu\big(\sigma^{n}(\bw)\big)}\big(\psi_{\ell,\omega_n}(\xi)+tk\big).
\end{aligned}
\]
It implies $\widehat{\mu\big(\sigma^{n}(\bw)\big)}\big(\psi_{\ell,\omega_n}(\xi)+tk\big)=0 $ for all $k\in\mathbb{Z}$, and we have $\psi_{\ell,\omega_n}(\xi)\in\mathcal{Z}_t(\mu(\sigma^n(\bw)))$.
	
Then, we show $\{\#X_n\}$ stabilises. For $\xi\in X_n$, we have 
\[
\xi = \frac{\xi_0 + t\ell_1 + tN_{\omega_1}\ell_2 + \cdots + tN_{\omega_1}\cdots N_{\omega_{n-1}}\ell_n}{N_{\omega_1}N_{\omega_2}\cdots N_{\omega_n}}.
\]
Since $\xi_0\in[0,t)$ and $0\le \ell_j<N_{\omega_j}$, we have $0\le\xi\le t$ and  $X_n\subseteq [-t,t]$ for all $n\ge 1$, and so 
\[
X_n \subseteq [-t,t] \cap \bigcup_{\boldsymbol{\eta}\in\Omega}\mathcal{O}\big(\widehat{\mu(\boldsymbol{\eta})}\big).
\]
By Lemma \ref{lem_zero-set-local-finite}, the set on the  right is  finite. Since $\{\#X_n\}$ is non-decreasing and bounded,   there exists $n_0>0$ such that $\#X_n=\#X_{n-1}$ for all $n>n_0$.

Finally, since  $\#\{k\in\N_+:\# B_{\omega_k}< N_{\omega_k}\}=\infty$, there exists $n>n_0$ such that $\#B_{\omega_n}<N_{\omega_n}$. Then for each $\xi\in X_{n-1}$, there exists a unique $\ell_0\in\{0,1,\dots,N_{\omega_n}-1\}$ such that $M_{B_{\omega_n}}(\psi_{\ell_0,\omega_n}(\xi))\neq 0$. Since
\[
\sum_{\ell\in L_{\omega_n}}\big|M_{B_{\omega_n}}(\psi_{\ell,\omega_n}(\xi))\big|^2=1,
\]
we   have $\ell_0\in L_{\omega_n}$ and 
\[
\left| M_{B_{\omega_{n}}}(\psi_{\ell_0,\omega_{n}}(\xi)) \right| = \Big| \frac{1}{\# B_{\omega_{n}}} \sum_{b \in B_{\omega_{n}}} e^{-2\pi i b \psi_{\ell_0,\omega_{n}}(\xi)} \Big| = 1.
\]
Since $0\in B_{\omega_n}$, we have $b\psi_{\ell_0,\omega_n}(\xi)\in\mathbb{Z}$ for all $b\in B_{\omega_n}$. For each $\ell\neq\ell_0$, we further obtain
\begin{align*}
\left| M_{B_{\omega_{n}}}\big(N_{\omega_{n}}^{-1}(t\ell-t\ell_0)\big) \right|
&=\Big| \frac{1}{\# B_{\omega_{n}}} \sum_{b \in B_{\omega_{n}}} e^{-2\pi i b \big(\psi_{\ell_0,\omega_{n}}(\xi)+N_{\omega_{n}}^{-1}(t\ell-t\ell_0)\big )} \Big|\\
&=\left| M_{B_{\omega_{n}}}(\psi_{\ell,\omega_{n}}(\xi)) \right|\\
&=0,
\end{align*}
which implies
\[
\big(t\{0,1,\dots,N_{\omega_n}-1\}-t\{0,1,\dots,N_{\omega_n}-1\}\big)\setminus\{0\}
\subset \mathcal{O}\big(\widehat{\delta}_{N_{\omega_n}^{-1}B_{\omega_n}}\big).
\]

Since $\gcd(t,N_{\omega_n})=1$,  we have 
$$
t\{0,1,\dots,N_{\omega_n}-1\}\equiv\{0,1,\dots,N_{\omega_n}-1\}\pmod {N_{\omega_n}},
$$
and  it follows that $\{e^{2\pi i\lambda x}:\lambda\in\{0,1,\dots,N_{\omega_n}-1\}\}$ forms an orthonormal set in $L^2(\delta_{N_{\omega_n}^{-1}B_{\omega_n}})$.  Since $(N_{\omega_n},B_{\omega_n},L_{\omega_n})$ is a Hadamard triple, we have
\[
\#B_{\omega_n}=\# L_{\omega_{n}}=\dim L^2(\delta_{N_{\omega_n}^{-1}B_{\omega_n}})\ge N_{\omega_n},
\]
which contradicts $\#B_{\omega_n}<N_{\omega_n}$. Hence the conclusion holds.         
\end{proof}

To prove that $\mu(\bw)^\bn$ is a $t$-equi-positive measure, we have to  consider the bounded and unbounded cases of the sequence $\mathbf{n}$ separately.
\begin{proposition}\label{pro_finite_equi-measure_bound}
Let $\mu^\mbf{n}$, $\mu$ and $\mathcal{E}_m$ be given by \eqref{def_mubn2}, \eqref{def_rm2} and \eqref{def_em2} respectively. Suppose $\mbf{n}$ is  bounded.
For all $\bw$ such that $\mu(\bw)$ is non-degenerate, $\mu(\bw)$ and $\mu^\mbf{n}(\bw)$ are $t$-equi-positive measures for all $t\in\mathcal{E}_m$.
\end{proposition}
\begin{proof}
Fix 	 $\bw\in\Omega$ such that $\mu(\bw)$ is non-degenerate.	We write
\begin{equation}\label{def_v>k}
\begin{split}
	\nu_{>k}(\bw)
	&=\delta_{N_{\omega_{k+1}}^{-1} B_{\omega_{k+1}}} * \delta_{(N_{\omega_{k+2}} N_{\omega_{k+1}})^{-1} B_{\omega_{k+2}}} * \cdots
	=\mu\big(\sigma^k(\bw)\big), \\
	\nu_{>k}(\bw)^{\mathbf{n}}
	&=\delta_{N_{\omega_{k+1}}^{-n_{k+1}} B_{\omega_{k+1}}} * \delta_{N_{\omega_{k+2}}^{-n_{k+2}} N_{\omega_{k+1}}^{-n_{k+1}} B_{\omega_{k+2}}} * \cdots
	=\mu\big(\sigma^k(\bw)\big)^{\sigma^k(\mathbf{n})},
\end{split}
\end{equation}
	where $\sigma^k(\mathbf{n})=\{n_j\}_{j=k+1}^{\infty}$.
	
	Since $\mu(\bw)$ is non-degenerate,  there exists a subsequence $\{\nu_{>k_j}(\bw)\}_{j=1}^\infty$ converging weakly to a measure singular to the Lebesgue measure. Since $\Omega$ is compact, $\{\sigma^{k_j}(\bw)\}$ has a convergent subsequence. For simplicity, we assume $\{\sigma^{m_j}(\bw)\}_{j=1}^\infty$ converges to some $\bu\in\Omega$. By Lemma \ref{lem_phicts}, the sequence $\{\nu_{>m_j}(\bw)\}_{j=1}^\infty$ converges weakly to $\mu(\bu)$, so $\mu(\bu)$ is singular to Lebesgue measure.
	
	Combining \eqref{def_v>k} with Lemma \ref{lem_singular}, we obtain the following fact. If $\mu(\bw)$ is absolutely continuous to the Lebesgue measure, then $\nu_{>k}(\bw)$ is absolutely continuous for every positive integer $k$; if $\mu(\bw)$ is singular to the Lebesgue measure, then $\nu_{>k}(\bw)$ is singular for all $k\in\mathbb{N}_+$.

	Next, we  prove that $\mu(\bw)$ is singular to Lebesgue measure. Suppose, on the contrary, $\mu $ is absolutely continuous. Then Lemma \ref{lem_singular} yields
	\[
	\#\bigl\{k\in\mathbb{N}_+:\# B_{\omega_k}< N_{\omega_k}\bigr\}<\infty.
	\]
Hence there exists $K_0$ such that $\#B_{\omega_k} = N_{\omega_k}$ for all $k\ge K_0$. 	Since $\sigma^{m_j}(\bw)\to\bu$, for each fixed $k>0$, there exists $j_0\in\mathbb{N}_+$ such that for all $j\ge j_0$,
\[
\omega_{m_j+1}\omega_{m_j+2}\dots \omega_{m_j+k}
= \alpha_1 \alpha_2 \dots \alpha_k .
\]
Choosing  a large $k$ so that $m_j+k\ge K_0$, we obtain $\#B_{u_k} = \#B_{\omega_{m_j+k}} = N_{u_k}$ for every $k\ge 1$.
By Lemma~\ref{lem_singular}, $\mu(\bu)$ is absolutely
continuous.  This contradicts the fact that $\mu(\bu)$ is singular.  Therefore
$\mu(\bw)$ is singular.

	Therefore, $\nu_{>k}(\bw)$ is singular for all $k$. Recall that
	\[
	\operatorname{cl}\bigl\{\nu_{>m_j}(\bw)\bigr\}_{j=1}^\infty
	=\bigl\{\nu_{>m_j}(\bw)\bigr\}_{j=1}^\infty \cup \{\mu(\bu)\}.
	\]
	By Proposition \ref{pro_finite_zero set}, $\{\nu_{>m_j}(\bw)\}_{j=1}^\infty$ forms a $t$-admissible family for every $t\in\mathcal{E}_m$ and by Theorem \ref{thm_admissible-equi-positive}, it is also a $t$-equi-positive family for all $t\in\mathcal{E}_m$. Hence $\mu(\bw)$ is a $t$-equi-positive measure for each $t\in\mathcal{E}_m$.
	
	We write
	\[
	\boldsymbol{\eta} = \omega_1^{n_1} \omega_2^{n_2} \omega_3^{n_3}\cdots \in\Omega,
	\]
that is, $\eta_j = \omega_k \quad \text{whenever } n_1 + \cdots + n_{k-1} < j \leq n_1 + \cdots + n_k. $
By Theorem \ref{thm_factorization}, we may decompose the infinite convolution measure $\mu(\boldsymbol{\eta})$ as
	\[
	\mu(\boldsymbol{\eta}) = \mu^{\mathbf{n}}(\bw) * \rho
	\]
	for some $\rho\in\mathcal{P}(\mathbb{R})$.   Then  for all $\xi\in\mathbb{R}$,  we have
	\begin{equation*}\label{eq_factor}
		\big| \widehat{\mu(\boldsymbol{\eta})}(\xi) \big|
		= \big| \widehat{\mu^{\mathbf{n}}(\bw)}(\xi)\cdot \widehat{\rho}(\xi) \big|
		\leq \big| \widehat{\mu^{\mathbf{n}}(\bw)}(\xi) \big|.
	\end{equation*}

For every $j\geq 1$, let 
\[
\boldsymbol{\eta}^{(m_j)} = \sigma^{\,\sum_{k=1}^{m_j} n_k}(\boldsymbol{\eta}) .
\]
The standard infinite convolution obtained by expanding each block in the definition of
$\nu_{>m_j}(\bw)^{\mathbf{n}}$ into single factors is precisely
$\mu(\boldsymbol{\eta}^{(m_j)})$. Therefore, by Theorem~\ref{thm_factorization},
there exists a probability measure $\rho_{m_j}$ such that
\[
\mu(\boldsymbol{\eta}^{(m_j)})
=
\nu_{>m_j}(\bw)^{\mathbf{n}} * \rho_{m_j}.
\]
Since $\bigl|\widehat{\rho_{m_j}}(\xi)\bigr|\le 1$, it follows that
\begin{equation}\label{eq_factor'}
\bigl|\widehat{\mu(\boldsymbol{\eta}^{(m_j)})}(\xi)\bigr|
\le
\bigl|\widehat{\nu_{>m_j}(\bw)^{\mathbf{n}}}(\xi)\bigr|
\end{equation}
for all $\xi\in\mathbb{R}$.

	Since $\mathbf{n}$ is  bounded and $\sigma^{m_j}(\bw)\to\bu$, there exists a positive integer sequence $\mathbf{n}'$ such that
	\[
	\bv = \alpha_1^{n_1'} \alpha_2^{n_2'} \alpha_3^{n_3'}\cdots,
	\]
	and there also exists a subsequence $\{\mu(\boldsymbol{\eta}^{(s_j)})\}_{j=1}^\infty$ of $\{\mu(\boldsymbol{\eta}^{(m_j)})\}_{j=1}^\infty$ converging weakly to $\mu(\bv)$. As $\mu(\bu)$ is singular to Lebesgue measure, so is $\mu(\bv)$. Using the arguments above, we conclude that $\{\mu(\boldsymbol{\eta}^{(s_j)})\}_{j=1}^\infty$ forms a $t$-equi-positive family for every $t\in\mathcal{E}_m$. Combined with inequality \eqref{eq_factor'}, the family $\{\nu_{>s_j}(\bw)^{\mathbf{n}}\}_{j=1}^\infty$ is also $t$-equi-positive for every $t\in\mathcal{E}_m$. This completes the proof that $\mu(\bw)^{\mathbf{n}}$ is a $t$-equi-positive measure for all $t\in\mathcal{E}_m$.
\end{proof}

\begin{proposition}\label{pro_finite_equi-measure_unbound}
Let $\mu^\mbf{n}$ and $\mathcal{E}_m$ be given by \eqref{def_mubn2} and \eqref{def_em2} respectively. Suppose $\mbf{n}$ is  unbounded.
For each $\bw\in\Omega$ , $\mu^\mbf{n}(\bw)$ is a $t$-equi-positive measure for all $t\in\N_+$.
\end{proposition}
\begin{proof}	
Since $\mathbf{n}$ is an unbounded sequence of positive integers, for every $t\in\mathbb{N}_+$, there exists a subsequence $\{n_{t_j}\}_{j=1}^\infty$ such that
\begin{equation}\label{eq_n_t_j}
	\frac{\max\big\{|b| : b\in B_k,\, 1\le k\le m\big\}}{\min\big\{N_k^{n_{t_j}} : 1\le k\le m\big\}} \le \frac{1}{32t}
\end{equation}
holds for all $j\in\mathbb{N}_+$.

Recall the infinite convolution representation:
\[
\nu_{>t_j-1}^{\mathbf{n}}(\bw)
= \delta_{N_{\omega_{t_j}}^{-n_{t_j}} B_{\omega_{t_j}}}
* \delta_{\big(N_{\omega_{t_j}}^{n_{t_j}} N_{\omega_{t_j+1}}^{n_{t_j+1}}\big)^{-1} B_{\omega_{t_j+1}}}
* \cdots
* \delta_{\big(\prod_{\ell=0}^s N_{\omega_{t_j+\ell}}^{n_{t_j+\ell}}\big)^{-1} B_{\omega_{t_j+s}}}
* \cdots .
\]
Taking the Fourier transform yields
\[
\widehat{\nu_{>t_j-1}^{\mathbf{n}}(\bw)}(\xi)
= \prod_{s=0}^{\infty} M_{B_{\omega_{t_j+s}}}\Big( \frac{\xi}{\prod_{\ell=0}^s N_{\omega_{t_j+\ell}}^{n_{t_j+\ell}}} \Big)
= \prod_{s=0}^{\infty} M_{B_{\omega_{t_j+s}}}(\xi_s),
\]
where we define
\[
\xi_s = \frac{\xi}{\prod_{\ell=0}^s N_{\omega_{t_j+\ell}}^{n_{t_j+\ell}}}, \quad \forall\, s\ge 0.
\]

For each $\xi \in [-1, t+1]$, it is clear that
\[
|\xi_s| \le \frac{1}{N_{\omega_{t_j}}^{n_{t_j}}} \cdot \frac{|\xi|}{2^s} \le \frac{t}{N_{\omega_{t_j}}^{n_{t_j}} \, 2^{s-1}}.
\]
  Combining this with inequality \eqref{eq_n_t_j}, we obtain that for each $b\in B_{\omega_{t_j+s}}$,
\[
|2\pi b \xi_s|
\le 2\pi |b| \cdot \frac{t}{N_{\omega_{t_j}}^{n_{t_j}} \, 2^{s-1}}
\le \frac{\pi}{2^{s+3}}. 
\]
Thus, we have
\begin{align*}
	\big| M_{B_{\omega_{t_j+s}}}(\xi_s) \big|
	&= \Big| \frac{1}{\#B_{\omega_{t_j+s}}} \sum_{b\in B_{\omega_{t_j+s}}} e^{-2\pi i b \xi_s} \Big| \\
	&=\Big| \frac{1}{\#B_{\omega_{t_j+s}}} \sum_{b\in B_{\omega_{t_j+s}}} e^{-2\pi i b \xi_s} \cdot e^{\frac{\pi i }{2^{s+3}}} \Big| \\
	&\ge \cos\left( \frac{\pi}{2^{s+2}} \right).
\end{align*}

Therefore, for all $\xi \in [-1, t+1]$,
\[
\widehat{\nu_{>t_j-1}^{\mathbf{n}}(\bw)}(\xi)
\ge \prod_{s=0}^{\infty} \cos\left( \frac{\pi}{2^{s+2}} \right) > 0.
\]
Hence $\{\nu_{>t_j-1}^{\mathbf{n}}(\bw)\}_{j=1}^\infty$ is a $t$-equi-positive family, and   $\mu^\bn(\bw)$ is a $t$-equi-positive measure.
\end{proof}

\begin{theorem}\label{thm_finite}
Let $\mu^\mbf{n}$, $\mu$ and $\mathcal{E}_m$ be given by \eqref{def_mubn2}, \eqref{def_rm2} and \eqref{def_em2} respectively.  
For all $\bw$ such that $\mu(\bw)$ is non-degenerate, $\mu^\mbf{n}(\bw)$ is a spectral measure, and there exist uncountably many subsets $\Lambda_{\bw}$ of $\Z$ such that $t\Lambda_{\bw}$ is a spectrum of $\mu^\mbf{n}(\bw)$ for all $t\in \mathcal{E}_m$.
\end{theorem}

\begin{proof}
This conclusion follows directly from Theorem \ref{thm_A-spectral eigen set}, Lemma \ref{lem_gcd-Hadamard}, Proposition \ref{pro_finite_equi-measure_bound} and Proposition \ref{pro_finite_equi-measure_unbound}.
\end{proof}

\begin{proof}[Proof of Theorem \ref{thm_finite_random}]
This conclusion follows directly from Theorem \ref{thm_finite}.
\end{proof}

\section{ Sufficient conditions for spectral random measures}

The hypotheses in our main theorems, especially the non-degeneracy condition, are not always easy to verify directly. In this section we therefore derive several explicit sufficient conditions.

\begin{corollary}\label{cor_finite_1}
Let $\mu^\mbf{n}$  and $\mathcal{E}_m$ be given by \eqref{def_mubn2}  and \eqref{def_em2} respectively.  Suppose that $\# B_k< N_k$ for $1 \leq k \leq m$. 
Then $\mu^\mbf{n}(\bw)$ is a spectral measure, and there exist uncountably many subsets $\Lambda_{\bw}$ of $\Z$ such that $t\Lambda_{\bw}$ is a spectrum of $\mu^\mbf{n}(\bw)$ for all $t\in \mathcal{E}_m$ for every $\bw \in \Omega$.
	
Moreover, all elements of $\mathcal{E}_m$ are spectral eigenvalues of $\mu^\mbf{n}(\bw)$ for every $\bw \in \Omega$.
\end{corollary}

\begin{proof}
Fix $\boldsymbol{w}\in\Omega$. Since $\Omega$ is compact, the sequence $\{\sigma^{n}(\boldsymbol{w})\}_{n=1}^\f$ has a convergent subsequence $\{\sigma^{n_j}(\boldsymbol{w})\}_{j=1}^\f$. Denote its limit by $\bu=\alpha_1\alpha_2\dots\in\Omega$. Then by Lemma \ref{lem_phicts}, $\mu(\sigma^{n_j}(\boldsymbol{w}))$ converges weakly to $\mu(\bu)$.

Since $\# B_k< N_k$ for $1 \leq k \leq m$, we obtain $\#\bigl\{k\in\mathbb{N}_+:\# B_{\alpha_k}< N_{\alpha_k}\bigr\}=\infty.$
By Lemma \ref{lem_singular}, $\mu(\bu)$ is singular to Lebesgue measure.
It follows that $\mu(\boldsymbol{w})$ is non‑degenerate.
Finally, with the non‑degeneracy of $\mu(\boldsymbol{w})$ established, we apply Theorem \ref{thm_finite} to complete the proof.
\end{proof}

\begin{proof}[Proof of Corollary \ref{cor_finite_random_1}]
	It follows immediately from  corollary \ref{cor_finite_1}.
\end{proof}

Similar results for the following lemma may be found in \cite{LMW-2023, LMZ25}. We provide a proof for readers' convenience.
\begin{lemma}\label{lem_convergence_of_word}
Let $\mathbb{P}$ be the Bernoulli measure on $\Omega$ given by \eqref{def_BPM} with respect to a positive  probability vector $\bp=(p_1,p_2,\ldots, p_m)$. Suppose that $p_1>0$. Given $\boldsymbol{\alpha}=111\cdots \in \Omega$, there exists $\Omega_{0}\subset \Omega $ with $\mathbb{P}(\Omega_0)=1$ such that for each $\boldsymbol{\omega} \in \Omega_{0}$,  we have that
$$
\lim_{j\to \infty } \sigma^{n_{j}}(\boldsymbol{\omega}) =\boldsymbol{\alpha},
$$
for some  strictly increasing  sequence $\{n_{j}\}_{j=1}^{\infty}$.
\end{lemma}

\begin{proof}
For each integer  $q\geq1$, choose a sequence $\{ k_{j}^{(q)}\}_{j=1}^{\infty}$ of positive integers  such that $k_{1}^{(q)} = 1$ and $k_{j+1}^{(q)} - k_{j}^{(q)} > q$ for all   $j\ge1$.
	
	Fix $q$. For each integer $i\geq 1$,  we define a random variable $X_i^{(q)}: \Omega \to \R$ by
	\begin{equation*}
		X_{i}^{(q)}(\boldsymbol{\omega})=
		\begin{cases}
			1, & \omega_{k_{i}^{(q)}} \omega_{k_{i}^{(q)}+1} \cdots \omega_{k_{i}^{(q)}+q-1} = 11\cdots 1, \\
			0, & \text{otherwise}.
		\end{cases}
	\end{equation*}
	Since the Bernoulli measure  $\mathbb{P}$  is generated by the probability vector $\bp=(p_1,p_2,\ldots, p_m)$ where $p_1>0$ , the expectation of $X_i^{(q)}$ is given by
	\begin{equation*}
		\mathbb{E}[X_{i}^{(q)}] = \mathbb{P}(X_{i}^{(q)} = 1) = p_{1} p_{1} \dots p_{1}=p_1^q>0,
	\end{equation*}
	for all $i\geq1$. Hence $\{ X_{i}^{(q)}(\boldsymbol{\omega}) \}_{i=1}^{\infty}$ is a sequence of independently identically distributed integrable random variables.
	
	By the  Kolmogorov strong law of large numbers, there exists a subset $\Omega_{q} \subseteq \Omega$ with $\mathbb{P}(\Omega_q)=1$ such that for all $\boldsymbol{\omega} \in \Omega_{q}$,
	\begin{equation*}
		\lim\limits_{n\to\infty} \dfrac{1}{n} \sum\limits_{i=1}^{n} X_{i}^{(q)}(\boldsymbol{\omega}) = E[X_{i}^{(q)}] = p_1^q,
	\end{equation*}
	which is equivalent to
	\begin{equation}\label{equation_Kolmogorov_strong_law_of_large_number_in_q_situation}
		\lim\limits_{n\to\infty} \dfrac{\# \{ 1\leq i \leq n : \omega_{k_{i}^{(q)}} \omega_{k_{i}^{(q)}+1} \cdots \omega_{k_{i}^{(q)}+q-1} = 11 \cdots 1 \}}{n} = p_1^q.
	\end{equation}

	Since $\mathbb{P}(\Omega_q)=1$ for all integers $q\geq 1$, we write
	\begin{equation*}
		\Omega_{0} = \bigcap_{q=1}^{\infty} \Omega_{q},
	\end{equation*}
	and it is clear that   $\mathbb{P}(\Omega_{0}) = 1.$  Hence for all $\boldsymbol{\omega} \in \Omega_{0}$ and all integers $q\geq 1$, we have that
	\begin{equation}\label{equation_Kolmogorov_strong_law_of_large_number}
		\lim\limits_{n\to\infty} \dfrac{\# \{ 1\leq i \leq n : \omega_{k_{i}^{(q)}} \omega_{k_{i}^{(q)}+1} \cdots \omega_{k_{i}^{(q)}+q-1} = 11 \cdots 1 \}}{n} = p_1^q>0.
	\end{equation}

	For each given  $\boldsymbol{\omega} =\omega_1\omega_2\cdots\in \Omega_{0}$, we define a sequence of integers $n_j$ inductively. First, for $q=1$, by \eqref{equation_Kolmogorov_strong_law_of_large_number}, there exists a sufficiently large integer $K_{1}\geq 1$ such that
	$$
	\# \{ 1\leq i \leq K_1 : \omega_{k_{i}^{(1)}} = 1\}>\frac{1}{2}K_1 \cdot p_1>1.
	$$
	We choose $i_1\in \{ 1\leq i \leq K_1 : \omega_{k_{i}^{(1)}} = 1\}$ and have $\omega_{k_{i_1}^{(1)}} = 1$. By setting $n_1=k_{i_1}^{(1)}-1$, we have that
	$$
	\sigma^{n_{1}}(\boldsymbol{\omega}) =1\omega_{n_1+2}\omega_{n_1+3}\cdots.
	$$
	
	Assume that the integers $\{ n_{j} \}_{j=1}^{l}$ have been chosen and  satisfy that   $n_{1} < n_{2} < \dots < n_{l}$ and
	\begin{equation}\label{shifpro}
		\sigma^{n_{j}}(\boldsymbol{\omega}) =1^j \ \omega_{n_{j}+j+1}\omega_{n_{j}+j+2}\cdots=11\cdots1\ \omega_{n_{j}+j+1}\omega_{n_{j}+j+2}\cdots.
	\end{equation}
	for all $1\leq j \leq l$. For  $q = l+1$, letting $\epsilon = \frac{1}{2}p_1^{l+1} $,  by \eqref{equation_Kolmogorov_strong_law_of_large_number}, there exists an  integer  $K_{l+1}\geq 1$ such that 	
	$$
	\# \{1\leq i \leq n: \omega_{k_{i}^{(l+1)}} \omega_{k_{i}^{(l+1)}+1} \cdots \omega_{k_{i}^{(l+1)}+l} = 11 \cdots 1=1^{l+1} \} \geq n \big(p_1^{l+1}  - \epsilon\big).
	$$
	for all $n>K_{l+1}$. For a sufficiently  large $n$, we choose $i_{l+1}$ such that $1\le i_{l+1}\le n$, $k_{i_{l+1}}^{(l+1)} > n_{{l}}+1$  and
	\begin{equation*}
		\omega_{k_{i_{l+1}}^{(l+1)}} \omega_{k_{i_{l+1}}^{(l+1)}+1} \cdots \omega_{k_{i_{l+1}}^{(l+1)}+l} = 1^{l+1}.
	\end{equation*}
	Setting  $n_{l+1} = k_{i_{l+1}}^{(l+1)}-1$,  we have that
	$$
	\sigma^{n_{l+1}}(\boldsymbol{\omega}) =1^{l+1}\omega_{n_{l+1}+l+2}\omega_{n_{l+1}+l+3}\cdots.
	$$
	
	Recall that  $\bu=111\cdots$. For each given $\boldsymbol{\omega}$, we obtain a sequence $\{n_j\}_{j=1}^\infty$ satisfying \eqref{shifpro}, and it follows  that
	$$
	d(\sigma^{n_{j}}(\boldsymbol{\omega}),\bu) \leq  2^{-j}
	$$
	for all $j\geq 1$. Therefore  $\{ \sigma^{n_{j}} (\boldsymbol{\omega}) \}$ converges to $\boldsymbol{\alpha} \in \Omega$,  and the conclusion holds.
\end{proof}

\begin{proof}[Proof of Theorem \ref{thm_finite_random_Bernoulli}]
By Theorem \ref{thm_finite_random}, it suffices to show that $\mu$ is non-degenerate $\mathbb{P}$-a.e.

Since $\mu$ is singular to Lebesgue measure $\mathbb{P}$-a.e., where $\mathbb{P}$ is the Bernoulli probability measure on $\Omega$ generated by a probability vector $\mathbf{p}=(p_1, p_2, \cdots, p_m)$, it follows from Lemma \ref{lem_singular} that there exists $1\le \alpha\le m$ such that $\# B_\alpha<N_\alpha$ and $p_\alpha>0$. Without loss of generality, we assume that $\alpha=1$.  Then, by Lemma \ref{lem_convergence_of_word}, there exists a subset $\Omega_0\subset \Omega$ with $\mathbb{P}(\Omega_0)=1$ such that for each $\boldsymbol{\omega} \in \Omega_0$, we have
$$
\lim_{j\to \infty } \sigma^{n_j}(\boldsymbol{\omega}) =\boldsymbol{\alpha}=111\cdots,
$$
for some strictly increasing sequence $\{n_j\}_{j=1}^{\infty}$. Thus, by Lemma \ref{lem_phicts}, $\mu(\sigma^{n_j}(\boldsymbol{\omega}))$ converges weakly to $\mu(\boldsymbol{\alpha})$. Since $\# B_\alpha<N_\alpha$, by Lemma \ref{lem_singular}, we have that $\mu(\boldsymbol{\alpha})$ is singular  to Lebesgue measure. It follows that $\mu$ is non-degenerate $\mathbb{P}$-a.e. This completes the proof.
\end{proof}	
	
The following result from \cite{Chen-Liu-Liu-2026} gives a necessary condition for a real number to be a spectral eigenvalue in the consecutive‑digit case, which simplifies our  proof of Theorem \ref{thm_finite_consecutive}.	
\begin{theorem} \label{thm_consecutive}
Let $\{(N_k,B_k)\}_{k=1}^\infty$ be a sequence of admissible pairs in $\mathbb{R}$, and suppose that  for each $k\in\N_+$, $B_k = \{0,1,\dots,b_k-1\}$ with $b_k\in\mathbb{N}_+$.
Let $\mu$ be given by \eqref{infinite-convolution}. Then $t\in\mathbb{R}$ is a spectral eigenvalue of $\mu$ only if $t=\dfrac{q}{p}$ with
\[
\gcd(p,q)=\gcd(p,b_k)=\gcd(q,b_k)=1
\]
for all $k\in\N_+$.
\end{theorem}

The following simple fact is useful in the proof of Theorem \ref{thm_finite_consecutive}. 
\begin{lemma}\label{lem_zsc}
Let $B = \{0,1,\dots ,b-1\}$. Then 
\[
\mathcal{O}(M_B)=\Bigl\{\frac{m}{b}\;:\; m\in\mathbb Z\setminus b\Z\Big\}.
\]
\end{lemma}

\begin{proof}
Since    $M_B(\xi)=\frac1b\sum_{j=0}^{b-1}e^{-2\pi i j\xi}$, we have
\[
b(1-e^{-2\pi i\xi})M_B(\xi)=1-e^{-2\pi i b\xi}.
\]
For $\xi\in\mathbb Z$, we have $M_B(\xi) =1.$ Therefore $M_B(\xi)=0$ if and only if $b\xi\in\mathbb Z$ but
$\xi\notin\mathbb Z$, i.e.\ $\xi=\frac{m}{b}$ with $m\in\mathbb Z\setminus b\Z$.
\end{proof}

\begin{proof}[proof of Theorem \ref{thm_finite_consecutive}]	
By Theorem \ref{thm_consecutive}, it suffices to prove that each $t\in\mathcal{E}_{\bw}$ is a spectral eigenvalue of $\mu^{\mathbf{n}}(\bw)$ for every $\bw\in\Omega$ with $\mu(\bw)$ non‑degenerate. It is easy to check that $t\in\mathcal{E}_{\bw}$ is an admissible eigenvalue of every $(N_{\omega_k},B_{\omega_k})$, then by Theorem \ref{thm_A-spectral eigen set}, we only need to verify that $\mu^{\mathbf{n}}(\bw)$ is a $t$-equi‑positive measure. 

If $\bn$ is unbounded, the conclusion follows from Proposition \ref{pro_finite_equi-measure_unbound}.

If $\bn$ is bounded, by the proof of Proposition \ref{pro_finite_equi-measure_bound}, it suffices to show that
\[
\mathcal{Z}_t\big(\mu(\bw)\big)=\emptyset
\]
for all $t\in \mathcal{E}_{\bw}$ and all $\bw \in \Omega$ satisfying
\[
\#\big\{k\in\mathbb{N}_+:\# B_{\omega_k}< N_{\omega_k}\big\}=\f.
\]

Assume that there exist $t\in \mathcal{E}_{\bw}$ and $\bw \in \Omega$ with $\#\{k\in\N_+:\# B_{\omega_k}< N_{\omega_k}\}=\infty$ such that $\mathcal{Z}_t\big(\mu(\bw)\big)\ne\emptyset$. By Lemma \ref{lem_zero-set-local-finite}, the set $[-t, t] \cap \left( \bigcup_{\boldsymbol{\eta} \in \Omega} \mathcal{O}\big(\widehat{\mu(\boldsymbol{\eta})}\big) \right)$ is finite. By the same construction as in Proposition~\ref{pro_finite_zero set},  $X_n$ is increasing and bounded, and there exists $n_0 > 1$ such that $\# X_{n-1} = \# X_{n}$ for all $n > n_0$.

Since $(N_k,B_k)$ is an admissible pair and $B_k = \{0,1,\dots,b_k-1\}$ for every $1\le k\le m$, it follows that $b_k \mid N_k$. Write $d_k = \frac{N_k}{b_k}$ and $L_k=\{0,d_k,2d_k,\cdots(b_k-1)d_k\}$. Then $(N_k,B_k,L_k)$ is a Hadamard triple.

Fix an integer $n>n_0$. For every $\xi \in X_{n-1}$, there exists a unique $\ell_0 \in \{0, 1, \dots, N_{\omega_{n}} - 1\}$ such that $M_{B_{\omega_{n}}}(\psi_{\ell_0,\omega_{n}}(\xi)) \neq 0$. Write $z=\psi_{\ell_0,\omega_{n}}(\xi)$. Since $z\in X_{n}$, by Lemma \ref{lem_zsc}, we get
\begin{equation}\label{inclzkt}
z+kt\in \mathcal{O}\big(\widehat{\mu\big(\sigma^{n}(\bw)\big)}\big)=\bigcup_{j=1}^\f N_{\omega_{n+1}}N_{\omega_{n+2}}\cdots N_{\omega_{n+j-1}} d_{\omega_{n+j}}\big(\mathbb{Z}\setminus b_{\omega_{n+j}}\mathbb{Z}\big),
\end{equation}
for all $k\in\Z$.  Thus, $z\in\Z$, and
\begin{align*}
	\left| M_{B_{\omega_{n}}}\big(N_{\omega_{n}}^{-1}(t\ell-t\ell_0)\big) \right|
	&=\Big| \frac{1}{\# B_{\omega_{n}}} \sum_{b \in B_{\omega_{n}}} e^{-2\pi i b \big(\psi_{\ell_0,\omega_{n}}(\xi)+N_{\omega_{n}}^{-1}(t\ell-t\ell_0)\big )} \Big|\\
	&=\left| M_{B_{\omega_{n}}}(\psi_{\ell,\omega_{n}}(\xi)) \right|\\
	&=0,
\end{align*}
for each $\ell\in \{0, 1, \dots, N_{\omega_{n}} - 1\}\setminus \{\ell_0\}$, which implies that
\begin{equation}\label{eq_bi-zero'}
	\left(t\{0, 1, \dots, N_{\omega_{n}} - 1\}-t\{0, 1, \dots, N_{\omega_{n}} - 1\}\right)\setminus\{0\}\subset \mathcal{O}\left(\widehat{\delta}_{N_{\omega_{n}}^{-1}B_{\omega_{n}}}\right).
\end{equation}

Let $d = \gcd(t,N_{\omega_{n}})$. Since $\gcd(t,b_{\omega_{n}})=1$, we have $d\mid d_{\omega_{n}}$. Consequently,
\[
t\{0, 1, \dots, N_{\omega_{n}}-1\}\equiv\{0, d, \dots, \bigl(\tfrac{N_{\omega_{n}}}{d}-1\bigr)d\}\pmod{N_{\omega_{n}}}.
\]
By \eqref{eq_bi-zero'}, the family
\[
\bigl\{e_\lambda(x) = e^{-2\pi i \lambda x} \;\big|\; \lambda \in \bigl\{0, d, \dots, \bigl(\tfrac{N_{\omega_{n}}}{d}-1\bigr)d\bigr\}\bigr\}
\]
is orthonormal in $L^2(\delta_{N_{\omega_{n}}^{-1}B_{\omega_{n}}})$. This implies that $d=d_{\omega_{n}}$, that is $d_{\omega_{n}}= \gcd(t,N_{\omega_{n}})$. Otherwise, the cardinality of this orthonormal set is strictly larger than that of an orthonormal basis, which gives a contradiction.

 By \eqref{inclzkt},  for each $k$, there exists $j_k\in\N_+$ and $x_k\in\mathbb{Z}\setminus b_{\omega_{n+j_k}}\mathbb{Z}$, such that 
$$
z+kt=N_{\omega_{n+1}}N_{\omega_{n+2}}\cdots N_{\omega_{n+j_k-1}} d_{\omega_{n+j_k}}x_k=\prod_{s=1}^{j_k-1}{b_{\omega_{n+s}}}\cdot\prod_{s=1}^{j_k}{d_{\omega_{n+s}}}x_k.
$$
Since $\gcd(t,b_{\omega_k}) = 1$ for all $k\in\mathbb{N}_+$, for each $r\in\mathbb{N}_+$, there exists some integer
\[
k_r\in\Big\{0,1,\dots,\prod_{s=1}^{r-1}b_{\omega_{n+s}}-1\Big\}
\]
such that
\[
z+k_rt\equiv0\pmod{\prod_{s=1}^{r-1}b_{\omega_{n+s}}}.
\]
Since $z$ and $t$ are fixed, there exists $k_0\in\N_+$ such that $z+kt>0$ for all $k>k_0$, it follows that there exists $r_0\in\N_+$ such that $z+k_rt>0$ for all $r>r_0$.

Recall that $d_{\omega_{n+s}}= \gcd(t,N_{\omega_{n+s}})$ for all $s\in\mathbb{N}_+$. Since $\gcd(t,b_{\omega_k}) = 1$ for all $k$, we obtain
$\gcd(d_{\omega_{n+s}}, b_{\omega_{n+s}}) = 1$ for each $s$. We claim $j_{k_r} \ge r$. 
Suppose, for contradiction, that $j_{k_r} < r$ for some $r$.  Then the zero
$z+k_r t$ is of the form $N_{\omega_{n+1}}\cdots N_{\omega_{n+j-1}} d_{\omega_{n+j}} \,x$
with $j = j_{k_r} < r$ and $x$ not divisible by $b_{\omega_{n+j}}$.  Consequently
$$
\prod_{s=1}^{r-1}b_{\omega_{n+s}}\bigg| z+k_rt=\prod_{s=1}^{j-1}{b_{\omega_{n+s}}}\cdot\prod_{s=1}^{j}{d_{\omega_{n+s}}}x,
$$
that is 
$$
\prod_{s=j}^{r-1}b_{\omega_{n+s}}\bigg| \prod_{s=1}^{j}{d_{\omega_{n+s}}}x,
$$
which gives a contradiction.

Hence, for all $r\in\N_+$ satisfying $r>r_0$ and $z<\prod_{s=1}^{r-1}b_{\omega_{n+s}}$, we have that $x_{k_r}>1$ and then
\begin{equation*}
\prod_{s=1}^{r-1}b_{\omega_{n+s}}\cdot\prod_{s=1}^{r}d_{\omega_{n+s}}\le\prod_{s=1}^{j_r-1}b_{\omega_{n+s}}\cdot\prod_{s=1}^{j_r}d_{\omega_{n+s}}\le z+k_rt\le 2t\cdot\prod_{s=1}^{r-1}b_{\omega_{n+s}}.
\end{equation*}
Dividing through by $\prod_{s=1}^{r-1}b_{\omega_{n+s}}$, we obtain
\[
\prod_{s=1}^{r}d_{\omega_{n+s}}<2t
\]
for all large $r$. On the other hand,  $\#\big\{k\in\mathbb{N}_+:\# B_{\omega_k}< N_{\omega_k}\big\}=\f$ ensures that $\prod_{s=1}^{r}d_{\omega_{n+s}}\to\infty$ as $r\to\infty$. This leads to a contradiction,  and the conclusion   follows.
\end{proof}

\begin{proof}[Proof of Corollary \ref{cor_finite_random_2}]
This conclusion follows directly from Theorem \ref{thm_finite_consecutive}.
\end{proof}


\begin{thebibliography}{10}
	\bibitem{An-Dong-He-2022}
	L.-X. An, X.~H. Dong and X.-G. He,
	\newblock On spectra and spectral eigenmatrix problems of the planar Sierpinski measures,
	\newblock \emph{Indiana Univ. Math. J.} 71 (2022), no.~2, 913--952; MR4420109
	   
	
	\bibitem{An-Fu-Lai-2019}
	L.-X. An, X.-Y. Fu, C.-K. Lai,
	\newblock On spectral Cantor-Moran measures and a variant of Bourgain's sum of sine problem,
	\newblock \emph{Adv. Math.} 349 (2019), 84–124.
	
	
	\bibitem{Bil03}	P. Billingsley,	\newblock {\em Convergence of probability measures},	\newblock John Wiley \& Sons, Inc., Second edition, 1999.
	\bibitem{Bogachev07}
	\newblock V.~I. Bogachev, {\it Measure theory}, Vol.~1, 2, Springer, Berlin, 2007. 	
	
	
	
	\bibitem{Chen-Liu-2023}
	M.-L. Chen and J. Liu,
	\newblock On spectra and spectral eigenmatrices of self-affine measures on $\Bbb R^n$,
	\newblock \emph{ Bull. Malays. Math. Sci. Soc.} 46 (2023), no.~5, Paper No. 162, 18 pp.; MR4616096
	  
	
	\bibitem{Chen-Liu-Liu-2026}
	S. Chen, J. Liu and M. Liu,
	\newblock The spectral eigenvalues of a class of Moran measures with continuous digits,
	\newblock \emph{Complex Anal. Oper. Theory} 20 (2026), no.~2, Paper No. 55, 16 pp.; MR5027674
	
	\bibitem{Cinlar11}
	E.~Cinlar.
	\newblock {\emph{Probability and Stochastics.}}
	\newblock Springer New York, NY, 2011.   
	
	
	\bibitem{Dai-2012}
	X.-R. Dai,
	\newblock When does a Bernoulli convolution admit a spectrum?,
	\newblock \emph{ Adv. Math.} 231 (2012), no. 3-4, 1681–1693.
	
	
	\bibitem{Dai-He-Lau-2014}
	X.-R. Dai, X.-G. He, K.-S. Lau,
	\newblock On spectral $N$-Bernoulli measures,
	\newblock \emph{Adv. Math.} 259 (2014), 511–531.
	
	\bibitem{Dai-2016}
	X.-R. Dai,
	\newblock Spectra of Cantor measures,
	\newblock \emph{Math. Ann.} 366 (2016), no.~3-4, 1621--1647; MR3563247
	
	\bibitem{Deng-Chen-2021}
	Q.-R. Deng, J.-B. Chen,
	\newblock Uniformity of spectral self-affine measures,
	\newblock \emph{Adv. Math.} 380 (2021), Paper No. 107568, 17 pp.
	
	
	
	
	\bibitem{Dutkay-Haussermann-Lai-2019}
	D. Dutkay, J. Haussermann, C.-K. Lai,
	\newblock Hadamard triples generate self-affine spectral measures,
	\newblock \emph{Trans. Amer. Math. Soc.} 371 (2019), no. 2, 1439–1481.
	
	
	
	\bibitem{Falco03}
	K. Falconer,
	\newblock {\em Fractal geometry: Mathematical foundations and applications},
	\newblock John Wiley \& Sons, Ltd., Third edition, 2014.
	
	

	\bibitem{Fuglede-1974}
	B. Fuglede,
	\newblock Commuting self-adjoint partial differential operators and a group theoretic problem,
	\newblock \emph{J. Funct. Anal.} 16 (1974), 101–121.
	
	\bibitem{Fu-He-2018}
	Y.-S. Fu and L. He,
	\newblock Scaling of spectra of a class of random convolution on $\Bbb{R}$,
	\newblock \emph{J. Funct. Anal.} 273 (2017), no.~9, 3002--3026; MR3692329
	   
	
	\bibitem{Fu-He-Wen-2018}
	Y.-S. Fu, X.-G. He and Z.~X. Wen,
	\newblock Spectra of Bernoulli convolutions and random convolutions,
	\newblock \emph{J. Math. Pures Appl.} (9) 116 (2018), 105--131; MR3826550
	   
    \bibitem{GM22}
    Y. Gu,  J. J. Miao.
    \newblock {Dimensions of a class of self-affine Moran sets.}
    \newblock  \emph{ J. Math. Anal. Appl.} \textbf{513} (2022), 126210.
    
    \bibitem{GM} Y.~Gu and J.~J. Miao.
    \newblock Dimension theory of non-autonomous iterated function systems.
    \newblock arXiv 2309.08151, 2023.
	
	\bibitem{He-Tang-Wu-2019}
	X.-G. He, M.-W. Tang, Z.-Y. Wu,
	\newblock Spectral structure and spectral eigenvalue problems of a class of self-similar spectral measures,
	\newblock \emph{J. Funct. Anal.} 277 (2019), no. 10, 3688–3722.
	
	
	\bibitem{Jorgensen-Pedersen-1998}
	P. Jorgensen, S. Pedersen,
	\newblock Dense analytic subspaces in fractal $L^2$-spaces,
	\newblock\emph{ J. Anal. Math.} 75 (1998), 185–228.
	
	
	\bibitem{Kong-Li-Wang-2026}
	D. Kong, K. Li and Z. Wang,
	\newblock Rational points in Cantor sets and spectral eigenvalue problem for self-similar spectral measures,
	\newblock \emph{Forum Math.} 38 (2026), no.~4, 1243--1255; MR5017464
   
	
	
	\bibitem{Laba-2001}	
	I. Łaba,
	\newblock Fuglede's conjecture for a union of two intervals,
	\newblock \emph{Proc. Amer. Math. Soc.} 129 (2001), no. 10, 2965–2972.
	
	\bibitem{Laba-Wang-2002}
	I. Łaba, Y. Wang,
	\newblock On spectral Cantor measures,
	\newblock \emph{ J. Funct. Anal.} 193 (2002), no. 2, 409–420.
	
	
	
	\bibitem{Li-Wu-2022}
    J. Li and Z.-Y. Wu,
	\newblock On spectral structure and spectral eigenvalue problems for a class of self similar spectral measure with product form,
	\newblock \emph{Nonlinearity} 35 (2022), no.~6, 3095--3117; MR4443929
	
	
	\bibitem{Miao-2022}
	W. Li, J.J. Miao, Z. Wang,
	\newblock Weak convergence and spectrality of infinite convolutions,
	\newblock \emph{Adv. Math.}
	404 (2022), Paper No. 108425, 26 pp.
	
	\bibitem{LMW-2023}
	W. Li, J.J. Miao, Z. Wang,
	\newblock Spectrality of Random Convolutions Generated by Finitely Many Hadamard Triples,
	\newblock \emph{Nonlinearity}	37(2024),  Paper No. 015003, 21 pp.
	

	
	\bibitem{Liu-2024}
	J.-C. Liu, M. Liu, M.-W. Tang and S. Wu,
	\newblock On spectral eigenmatrix problem for the planar self-affine measures with three digits,
	\newblock \emph{Ann. Funct. Anal.} 15 (2024), no.~4, Paper No. 83, 22 pp.; MR4791481
	
	\bibitem{Liu-Tang-Wu-2023}
	J. Liu, M.-W. Tang and S. Wu,
	\newblock The spectral eigenmatrix problems of planar self-affine measures with four digits,
	\newblock \emph{Proc. Edinb. Math. Soc.} (2) 66 (2023), no.~3, 897--918; MR4637402
	  
	\bibitem{LMZ25} 
	 H. Y. Liu, J. J. Miao and H. B. Zhao,
	\newblock {Existence and Spectrality of random measures generated by infinite convolutions},
	\newblock {arXiv:2504.15744v1}.  
	
	
		
	\bibitem{Lu-2026}
	Z.-Y. Lu,
	\newblock The spectral eigenvalue set and Beurling dimension on self-similar measures,
	\newblock \emph{J. Math. Pures Appl.} (9) 205 (2026), Paper No. 103809, 21 pp.; MR4973614
	
	\bibitem{MZ24} 
	J. J. Miao and H. B. Zhao,
	\newblock {Existence  and spectrality of infinite convolutions generated by infinitely many admissible pairs},
	\newblock \emph{J. Fourier Anal. Appl.} 32 (2026), no.~1, Paper No. 12, 24 pp.; MR5013389


    \bibitem{MZ25} 
    J. J. Miao and H. B. Zhao,
    \newblock {Existence, equivalence and spectrality of infinite convolutions in $\R^d$},
    \newblock {	arXiv:2506.06670}.  

	
	\bibitem{MZ26} 
	J. J. Miao and H. B. Zhao,
	\newblock {Spectral properties of infinite convolutions generated by complete residue systems},
	\newblock \emph{Forum Math.} 38 (2026), no.~4, 1009--1026; MR5017452
	   
	
	

		
	\bibitem{Mattila-1995}
	P. Mattila,
	\newblock {\it Geometry of sets and measures in Euclidean spaces},
	\newblock Cambridge Studies in Advanced Mathematics, 44, Cambridge Univ. Press, Cambridge, 1995; MR1333890 
	
	
	
	
	
	\bibitem{BBT08}
	B.~S.~Thomson, J.~B.~Bruckner, A.~M.~Bruckner.
	\newblock{\em Real Analysis}, Second Edition.
	\newblock{www.classicalrealanalysis.com}, 2008.
	
	\bibitem{Walters-1982}
	P. Walters,
	\newblock \emph{An introduction to ergodic theory},
	\newblock Springer-Verlag, New York-Berlin, 1982.
	
	
	
	
\end{thebibliography}
\end{document}